\documentclass[a4paper,11pt,reqno,noindent]{amsart}
\usepackage[centertags]{amsmath}
\usepackage{amsfonts,amssymb,amsthm,dsfont,cases,amscd,esint,enumerate,bm}
\usepackage[T1]{fontenc}
\usepackage[english]{babel}
\usepackage[applemac]{inputenc}
\usepackage{newlfont}
\usepackage{color}
\usepackage[body={15cm,21.5cm},centering]{geometry} 
\usepackage{fancyhdr}
\usepackage{enumerate}

\theoremstyle{plain}

\newtheorem{theor}{Theorem}[section]
\newtheorem{question}{Question}[section]
\newtheorem{lem}[theor]{Lemma}
\newtheorem{prop}[theor]{Proposition}

\theoremstyle{definition}
\newtheorem{rems}[theor]{Remarks}
\newtheorem{rem}[theor]{Remark}

\mathchardef\emptyset="001F
\numberwithin{equation}{section}

\newcommand{\Sp}{\mathbb S}
\newcommand{\e}{\varepsilon}

\newcommand{\dist}{\operatorname{dist}}
\newcommand{\Q}{\mathbb Q}
\newcommand{\R}{\mathbb R}
\newcommand{\Z}{\mathbb Z}

\newcommand{\Ic}{\mathcal I}

\newcommand{\cvf}{\rightharpoonup}
\newcommand{\cvfs}{\overset*{\rightharpoonup}}

\newcommand{\loc}{{\operatorname{loc}}}
\newcommand{\Id}{\operatorname{Id}}

\newcommand{\E}{\mathbb{E}}

\newcommand{\per}{{\operatorname{per}}}

\newcommand{\Aa}{\boldsymbol a}
\newcommand{\bb}{\bar{\boldsymbol b}}

\newcommand{\Ld}{\operatorname{L}}

\newcommand{\Div}{{\operatorname{div}}}

\newcommand{\curl}{\operatorname{curl}}

\newcommand{\step}[1]{\noindent \textit{Step} #1.}
\newcommand{\substep}[1]{\noindent \textit{Substep} #1.}
\newcommand{\Pm}{\mathbb{P}}
\newcommand{\pr}[1]{\mathbb{P}\left[ #1 \right]}

\newcommand{\expec}[1]{\mathbb{E}\left[ #1 \right]}

\usepackage[colorlinks,citecolor=black,urlcolor=black]{hyperref}

\title[Homogenization of the 2D Euler system: lakes and porous media]
{Homogenization of the 2D Euler system:\\lakes and porous media}

\author[M. Duerinckx]{Mitia Duerinckx}
\address[Mitia Duerinckx]{
Universit\'e Libre de Bruxelles, D\'epartement de Math\'ematique, 1050~Brussels, Belgium}
\email{mitia.duerinckx@ulb.be}
\author[A. Gloria]{Antoine Gloria}
\address[Antoine Gloria]{Sorbonne Universit\'e, CNRS, Universit\'e de Paris, Laboratoire Jacques-Louis Lions, 75005~Paris, France \&   Universit\'e Libre de Bruxelles, D\'epartement de Math\'ematique, 1050~Brussels, Belgium}
\email{antoine.gloria@sorbonne-universite.fr}

\begin{document}
\selectlanguage{english}

\begin{abstract}
This work is devoted to the long-standing open problem of homogenization of 2D perfect incompressible fluid flows, such as the 2D Euler equations with impermeable inclusions modeling a porous medium, and such as the lake equations. The main difficulty is the homogenization of the transport equation for the associated fluid vorticity. In particular, a localization phenomenon for the vorticity could in principle occur, which would rule out the separation of scales. Our approach combines classical results from different fields to prevent such phenomena and to prove homogenization towards variants of the Euler and lake equations: we rely in particular on the homogenization theory for elliptic equations with stiff inclusions, on criteria for unique ergodicity of dynamical systems, and on complex analysis in form of extensions of the Rad\'o--Kneser--Choquet theorem.
\end{abstract}

\maketitle
\setcounter{tocdepth}{1}
\tableofcontents

\section{Introduction}
 This work is devoted to the homogenization of perfect incompressible fluid flows described by the 2D Euler equations in several heterogeneous settings. Starting from the vorticity formulation, the question naturally splits into the homogenization of the div-curl problem defining the fluid velocity and the homogenization of the transport equation for the fluid vorticity. Due to heterogeneities, the fluid velocity typically has large high-frequency oscillations around its effective behavior. Due to this multiscale structure of the velocity field, the homogenization of the corresponding transport equation for the vorticity is a highly delicate question, and a well-defined limit equation might in general not exist, cf.~\cite[Section~3.2]{Jabin-Tzavaras-09}. Yet, the fluid velocity is not arbitrary: it is given as the solution of a div-curl problem, which amounts to an elliptic problem with the vorticity as a source term. Leveraging on this specific structure of the fluid velocity, we manage to homogenize the transport equation for the vorticity in certain circumstances. As we shall see, the main difficulty is to avoid the possibility of a localization phenomenon, and to show that the vorticity cannot be trapped due to heterogeneities. This is solved by using the invertibility of the relevant corrected harmonic coordinates adapted to the heterogeneities, which is obtained as a consequence of the work of Alessandrini and Nesi~\cite{Alessandrini-Nesi-01,Alessandrini-Nesi-18} --- an extension of the Rad\'o--Kneser--Choquet theorem.
We focus on the following two model problems:
\begin{enumerate}[\quad(1)]
\item homogenization of the 2D Euler system with impermeable inclusions modeling a porous medium;
\smallskip\item homogenization of the lake equations.
\end{enumerate}
Our main results for those models are described in the upcoming two sections in the periodic case, cf.~Sections~\ref{sec:inclusions}--\ref{sec:lakes}. In Section~\ref{sec:random}, we further comment on the random setting, which leads to some surprisingly delicate probabilistic open questions.
Note that these homogenization questions of the 2D Euler equations have a similar flavor as previous work on the 2D Euler equations with oscillatory initial data, see e.g.~\cite{DiPerna-Majda-87,Cheverry-06,Cheverry-Gues-08,Cheverry-08}.

\subsection{Homogenization of impermeable inclusions}\label{sec:inclusions}
We study the homogenization limit of the 2D Euler equations in a porous medium made of small impermeable holes. We start by describing the setting. We assume that the porous medium is characterized on the microscopic scale by an infinite collection of inclusions $\{I_n\}_n$ in $\R^2$, where each $I_n\subset\R^2$ is a connected open set.
Given $\e>0$ denoting the microscopic scale, and given an initial velocity field $u_\e^\circ\in \Ld^2(\R^2)^2$ with \mbox{$\Div(u_\e^\circ)=0$} and $\curl(u_\e^\circ)\in\Ld^1\cap\Ld^\infty(\R^2)$, we consider the unique global weak solution $u_\e\in\Ld^\infty_\loc(\R_+;\Ld^2(\R^2)^2)$ with $\curl(u_\e)\in\Ld^\infty_\loc(\R_+;\allowbreak\Ld^1\cap\Ld^\infty(\R^2))$ of the Euler equations with $\e$-rescaled impermeable inclusions at $\Ic_\e:=\e\Ic:=\bigcup_n\e I_n$,
\begin{equation}\label{eq:imperm-main}
\left\{\begin{array}{ll}
\partial_t u_\e+u_\e\cdot\nabla u_\e+\nabla p_\e=f_\e,&\text{in $\R^2\setminus\Ic_\e$},\\
\Div(u_\e)=0,&\text{in $\R^2\setminus\Ic_\e$},\\
u_\e\cdot\nu=0,&\text{on $\partial\Ic_\e$},\\
\int_{\e\partial I_n}u_\e\cdot\tau=0,&\text{for all $n$},\\
u_\e|_{t=0}=u_\e^\circ,
\end{array}\right.
\end{equation}
where $f_\e\in \Ld^\infty_\loc(\R_+;\Ld^2(\R^2)^2)$ is a given force field, which we assume to satisfy $\curl(f_\e)\in \Ld^\infty_\loc(\R_+;\allowbreak\Ld^1\cap\Ld^\infty(\R^2))$ and $f_\e\cdot\tau=0$ on $\partial\Ic_\e$ ($\nu$ is the outward unit normal vector and $\tau=\nu^\bot$ is the tangent vector at the boundary of the inclusions).
We refer to~\cite{Yudovich-63,Kikuchi,Majda-Bertozzi} for the well-posedness theory.
The problem of the macroscopic limit $\e\downarrow0$ for this system was first considered in~\cite{Mikelic-Paoli-99,Lions-Masmoudi-05}, where the weakly nonlinear regime was treated by means of two-scale convergence methods.
Beyond this particular regime, the problem has drawn quite a lot of attention but remains relatively poorly understood.
The critical scaling for the size of the inclusions was identified in~\cite{Noel-Lacave-Masmoudi-15,Lacave-Masmoudi-16}: denoting by $d_\e$ the typical minimal distance between inclusions and by $r_\e$ their typical diameter,  in the dilute regime~${r_\e}/{d_\e}\to0$ the fluid is not perturbed by the porous medium, while in the concentrated regime~\mbox{${r_\e}/{d_\e}\to\infty$} the fluid does not penetrate the porous medium. In the critical regime~$r_\e\sim d_\e$, a nontrivial effective modification of the Euler equations is expected to emerge. This was recently studied in~\cite{Hillairet-Lacave-Wu-22} under two important restrictions:
\begin{enumerate}[---]
\item The fluid vorticity in~\cite{Hillairet-Lacave-Wu-22} is assumed to be supported away from the porous medium: this strongly simplifies the problem as it allows to use mere weak compactness to pass to the limit in the transport equation for the vorticity. Circumventing this restriction is the main motivation of our work.
\smallskip\item The analysis in~\cite{Hillairet-Lacave-Wu-22} is based on the method of reflections to approximate the fluid velocity: for that reason, the result is valid in some perturbative diagonal regime~${r_\e}\ll d_\e$ only, when the inclusions have a vanishing volume fraction.
\end{enumerate}
In the present work, we relax both restrictions and obtain the first homogenization result in the critical regime $r_\e\sim d_\e$.
We start with the periodic case and we refer to Section~\ref{sec:random} below for the random setting.
\begin{theor}[Homogenization of impermeable inclusions]\label{th:imper}
Let $\{I_n\}_n$ be a periodic collection of subsets of $\R^2$
satisfying the following two conditions for some $\rho>0$:
\begin{enumerate}[---]
\item \emph{Regularity:} For all $n$, the inclusion $I_n$ is open, connected, has diameter bounded by $\frac1\rho$, and satisfies interior and exterior ball conditions with radius~$\rho$.
\item \emph{Hardcore condition:} For all $n\ne m$, we have $\dist(I_n,I_m)\ge\rho$.
\end{enumerate}
For all $\e>0$, let~$u_\e$ be the unique solution of the 2D Euler system~\eqref{eq:imperm-main} with $\e$-rescaled impermeable inclusions at $\Ic_\e=\e\Ic=\bigcup_n\e I_n$, and let the data of the problem converge in the macroscopic limit $\e\downarrow0$ in the sense of
\begin{equation}\label{eq:conv-data-imperm}
\begin{array}{rcll}
\curl(u_\e^\circ)\mathds1_{\R^2\setminus\Ic_\e}&\cvfs&w^\circ,&\text{in $\Ld^1\cap\Ld^\infty(\R^2)$,}\\
f_\e\mathds1_{\R^2\setminus\Ic_\e}&\cvfs&f,&\text{in $\Ld^\infty_\loc(\R_+;\Ld^2(\R^2))$,}\\
\curl(f_\e)\mathds1_{\R^2\setminus\Ic_\e}&\cvfs&g,&\text{in $\Ld^\infty_\loc(\R_+;\Ld^1\cap\Ld^\infty(\R^2))$.}
\end{array}
\end{equation}
By compactness, we then have, along a subsequence,
\begin{equation}\label{eq:lim-omu-res}
\begin{array}{rcll}
\curl(u_\e)\mathds1_{\R^2\setminus\Ic_\e}&\cvfs&w,&\text{in $\Ld^\infty_\loc(\R_+;\Ld^1\cap\Ld^\infty(\R^2))$},\\
u_\e\mathds1_{\R^2\setminus\Ic_\e}&\cvfs& u,&\text{in $\Ld^\infty_\loc(\R_+;\Ld^2(\R^2))$},
\end{array}\end{equation}
for some limit point $(w,u)$.
Assume that the latter is nowhere one-dimensional with rational direction: more precisely, we assume that it satisfies the following non-degeneracy condition,
\begin{gather}
\nexists\, U\subset\R_+\times\R^2~\text{open},~~\nexists\, e=(e_1,e_2)\in\Sp^1~\text{with $e_1/e_2\in\Q$},\nonumber\\
\text{such that $u(t,x)=e^\bot u_0(t,e\cdot x)$ on $U$ for some $u_0$,}\label{eq:non-deg-re}\\
\text{and $w(t,x)=w_0(t,e\cdot x)$ on $U$ for some $w_0$.}\nonumber
\end{gather}
Then, $(w,u)$ is the unique global weak solution of the following homogenized system in~$\R^2$,
\begin{equation}\label{eq:homog-imperm}
\left\{\begin{array}{l}
\partial_tw+\tfrac1{1-\lambda}\Div(w u)=g,\\
\Div(u)=0,\\
\curl(\bm{\bar m}u)=w,\\
w|_{t=0}=w^\circ,
\end{array}\right.
\end{equation}
where $\lambda:=|\Ic\cap Q|$ is the inclusions' volume fraction in the periodicity cell $Q=[-\frac12,\frac12)^2$ and
where $\bm{\bar m}$ is the symmetric matrix given by
\begin{equation}\label{eq:def-barm}
\bm{\bar m}\,:=\,J^T\bm{\bar a}J\in\R^{2\times2},\qquad J:=\begin{pmatrix}0&-1\\1&0\end{pmatrix},
\end{equation}
in terms of the homogenized matrix $\bm{\bar a}\in\R^{2\times2}$ defined componentwise by
\begin{equation}\label{eq:def-bara}
\bm{\bar a}_{ij}\,:=\,\delta_{ij}+\int_Q\nabla\varphi_i\cdot\nabla\varphi_j\,=\,\int_Q(\nabla\varphi_i+e_i)\cdot(\nabla\varphi_j+e_j),\qquad1\le i,j\le 2,
\end{equation}
where for~\mbox{$1\le i\le2$} the corrector $\varphi_i\in H^1_\per(\R^2)$ is the unique mean-zero periodic solution of
\begin{equation}\label{eq:varphi-imperm}
\left\{\begin{array}{ll}
-\triangle\varphi_i=0,&\text{in $\R^2\setminus\cup_nI_n$},\\
\nabla\varphi_i+e_i=0,&\text{in $\cup_nI_n$},\\
\int_{\partial I_n}\partial_\nu\varphi_i=0,&\text{for all $n$}.
\end{array}\right.
\end{equation}
\end{theor}
Note that this result is only conditional in the sense that we need the non-degeneracy condition~\eqref{eq:non-deg-re} on the weak limit point $(w,u)$ to avoid problematic directions.
As shown in the proof, this condition is crucial to ensure the well-posedness of the homogenization problem for the vorticity transport equation.
Although it is unclear to us how this condition could be replaced by an assumption on the data $u_\e^\circ,f_\e$, it should be viewed as a genericity condition.
\begin{rems}\label{rems:imperm}
A few comments are in order.
\begin{enumerate}[(a)]
\item \emph{Limit equations:}\\
As the homogenized matrix $\bm{\bar a}$ is not scalar in general, the limit equations~\eqref{eq:homog-imperm} do not take the form of the standard 2D Euler equations: the limit vorticity $\omega$ is transported by $\frac1{1-\lambda}u$, where $u$ is the limit velocity, but the div-curl problem defining the velocity from the vorticity is solved by the \emph{deformed} Biot--Savard law
\[u=\nabla^\bot(\Div\bm{\bar a}\nabla)^{-1}w,\]
instead of the usual $\nabla^\bot\triangle^{-1}w$ for 2D Euler.

\smallskip
\item \emph{Dilute expansion:}\\
If inclusions are dilute, we can make the limit equations more explicit by appealing to an effective medium expansion in form of the Clausius--Mossotti formula. Under suitable dilution assumptions as described in~\cite{DG-16a,DG-Einstein,DG-CM}, denoting by $0\le \lambda<1$ the volume fraction of the inclusions, if each inclusion $I_n$ is the translation of a given solid~$I_0$, we find
\begin{equation}\label{eq:Einstein-imperm}
\bm{\bar a}\,=\,\Id+\lambda\bm{\bar a}^0+o(\lambda),
\end{equation}
where $\bm{\bar a}^0$ is given by
\[\bm{\bar a}^0_{ij}\,:=\,|I_0|^{-1}\int_{\R^2}\nabla\varphi_i^0\cdot\nabla\varphi_j^0\,=\,|I_0|^{-1}\int_{\partial I_0}x_j\,\partial_\nu(\varphi_i^0+x_i),\]
in terms of the single-inclusion corrector problem
\begin{equation*}
\left\{\begin{array}{ll}
-\triangle\varphi_i^0=0,&\text{in $\R^2\setminus I_0$},\\
\nabla\varphi_i^0+e_i=0,&\text{in $I_0$},\\
\int_{\partial I_0}\partial_\nu\varphi_i^0=0,\\
\varphi_i^0\in\dot H^1(\R^2).\\
\end{array}\right.
\end{equation*}
Combined with this dilute approximation~\eqref{eq:Einstein-imperm}, the limit equations~\eqref{eq:homog-imperm} coincide with the effective equations derived in~\cite{Hillairet-Lacave-Wu-22} (to the exception that we consider here a periodic porous medium that fills the whole space, hence $\bm{\bar a}^0$ is constant, while in~\cite{Hillairet-Lacave-Wu-22} the domain of the porous medium is a compact set that does not intersect the support of the vorticity).
In case of spherical inclusions, the single-inclusion corrector problem can be solved explicitly,  which leads to $\bm{\bar a}^0=2\Id$, cf.~\cite{DG-16a,DG-Einstein,DG-CM}.

\smallskip\item \emph{Extensions:}\\
$\bullet$ We assume that the set of inclusions $\{I_n\}_n$ describing the porous medium is periodic, but the proof is easily adapted to a locally periodic setting. More precisely, this amounts to assuming $\mathds1_{\Ic_\e}(x)=\chi(x,\frac x\e)$ where $\chi(x,\cdot)$ is periodic for all $x\in\R^2$. In that setting, we note that the homogenized matrix $\bm{\bar a}$ would further depend on the macroscopic variable $x$. This allows to cover in particular the setting considered in~\cite{Hillairet-Lacave-Wu-22}, where the porous medium only fills a given compact subset of the domain.
\\[1mm]
$\bullet$ We consider for simplicity the Euler equations in the whole plane $\R^2$, but the proof is easily adapted to a bounded domain with suitable boundary conditions.
\\[1mm]
$\bullet$ A similar result can be obtained if periodicity is replaced by quasiperiodicity, in which case such point processes are called (Delone) quasilattices; see e.g.~the review article~\cite{MR1668082}. In that case, in the non-degeneracy condition~\eqref{eq:non-deg-re}, rational directions $e=(e_1,e_2)\in\Sp^1$ with $e_1/e_2\in\Q$ need to be replaced by directions $e\in\Sp^1$ such that 
$e^\bot \R\cap F\Z^M\ne\{0\}$, where $F\in\R^{2\times M}$ stands for the winding matrix of the quasiperiodic structure. 
\end{enumerate}
\end{rems}
We briefly comment on corresponding homogenization questions for {\it viscous} fluid flows in porous media. The most natural model is the Navier--Stokes equations in a perforated domain with no-slip boundary conditions. In case of a fixed positive viscosity, homogenization is by now well-understood and leads to the Navier--Stokes equations, to Darcy's law, or to the Navier--Stokes--Brinkman equations, depending on the regime between the typical size of the inclusions and their distances. In contrast, the case of vanishing viscosity is much more intricate due to viscous boundary layers, and we refer to~\cite{Lacave-Mazzucato-16,Hofer-23} and the references therein.
As no-slip boundary conditions do not lead to a useful vorticity formulation, the approach of the present work is of no use in that case.

\subsection{Homogenization of the 2D lake equations}\label{sec:lakes}
We turn to the homogenization problem for the lake equations, which take form of the 2D Euler equations with a modified divergence constraint $\Div(bu)=0$, where $b$ encodes the effect of the bottom topography in a shallow-water limit, cf.~\cite{Levermore-Oliver-Titi-1}. Let the bottom topography be described on the microscopic scale by some scalar depth function $b:\R^2\to\R_+$ that is assumed to be uniformly non-degenerate in the sense of
\begin{equation}\label{eq:unifell}
\tfrac1{C_0}\,\le\,b(x)\,\le\,C_0,\qquad\text{for all $x\in\R^2$},
\end{equation}
for some $C_0>0$.
Given $\e>0$ denoting the microscopic scale, and given an initial velocity field $u_\e^\circ\in \Ld^2(\R^2)^2$ with $\Div(b(\tfrac\cdot\e)u_\e^\circ)=0$ and $\curl(u_\e^\circ)\in\Ld^1\cap\Ld^\infty(\R^2)$, we consider the unique global weak solution $u_\e\in\Ld^\infty_\loc(\R_+;\Ld^2(\R^2)^2)$ with $\curl(u_\e)\in\Ld^\infty_\loc(\R_+;\Ld^1\cap\Ld^\infty(\R^2))$ of the lake equations with oscillatory depth $b(\frac\cdot\e)$,
\begin{equation}\label{eq:lake}
\left\{\begin{array}{ll}
\partial_t u_\e+u_\e\cdot\nabla u_\e+\nabla p_\e=f_\e,\\
\Div(b(\tfrac\cdot\e) u_\e)=0,\\
u_\e|_{t=0}=u_\e^\circ,
\end{array}\right.
\end{equation}
where $f_\e\in \Ld^\infty_\loc(\R_+;\Ld^2(\R^2)^2)$ is a given force field, which we assume to satisfy $\curl(f_\e)\in \Ld^\infty_\loc(\R_+;\Ld^1\cap\Ld^\infty(\R^2))$.
We refer to~\cite{Levermore-Oliver-Titi-1,Levermore-Oliver-Titi-2,D-17} for the well-posedness theory for the lake equations in the non-degenerate setting~\eqref{eq:unifell}.
Note that the 2D Euler equations~\eqref{eq:imperm-main} with impermeable inclusions can be seen as a particular case of the lake equations in the degenerate setting when the depth function~$b$ would be sent to $0$ inside the inclusions: impermeable inclusions amount to islands with vanishing depth. This analogy will actually be useful in Section~\ref{sec:harmcoord}.
Homogenization of the lake equations was first considered in~\cite{Bresch-GV-07} in the weakly nonlinear regime, but to our knowledge no result has been obtained beyond this particular case.
By a similar analysis as for the 2D Euler equations with impermeable inclusions, we obtain the following result.

\begin{theor}[Homogenization of the lake equations]\label{th:lake}
Let $b:\R^2\to\R_+$ be a periodic function that satisfies~\eqref{eq:unifell} and that is locally H\"older-continuous.
For all $\e>0$, let~$u_\e$ be the unique global weak solution of the lake equations~\eqref{eq:lake} with $\e$-rescaled depth function~$b(\frac\cdot\e)$, and let the data of the problem converge in the macroscopic limit $\e\downarrow0$ in the sense of
\begin{equation}\label{eq:conv-data-lake}
\begin{array}{rcll}
\curl(u_\e^\circ)&\cvfs&w^\circ,&\text{in $\Ld^\infty(\R^2)$,}\\
f_\e&\cvfs&f,&\text{in $\Ld^\infty_\loc(\R_+;\Ld^2(\R^2))$,}\\
\curl(f_\e)&\cvfs&g,&\text{in $\Ld^\infty_\loc(\R_+;\Ld^1\cap\Ld^\infty(\R^2))$.}
\end{array}
\end{equation}
By compactness, we have, along a subsequence,
\begin{equation}\label{eq:conv-res-lake}
\begin{array}{rcll}
\curl(u_\e)&\cvfs&w,&\text{in $\Ld^\infty(\R_+;\Ld^1\cap\Ld^\infty(\R^2))$},\\
u_\e&\cvfs&u,&\text{in $\Ld^\infty_\loc(\R_+;\Ld^2(\R^2))$},\\
b(\tfrac\cdot\e)u_\e&\cvfs&v,&\text{in $\Ld^\infty_\loc(\R_+;\Ld^2(\R^2))$}.
\end{array}
\end{equation}
for some limit point $(w,u,v)$. Assume that the latter is nowhere one-dimensional with rational direction: more precisely, we assume that the pair $(w,v)$ satisfies the following non-degeneracy condition,
\begin{gather}
\nexists\, U\subset\R_+\times\R^2~\text{open},~~\nexists\, e=(e_1,e_2)\in\Sp^1~\text{with $e_1/e_2\in\Q$},\nonumber\\
\text{such that $v(t,x)=e^\bot v_0(t,e\cdot x)$ on $U$ for some $v_0$,}\label{eq:non-deg-re-lake}\\
\text{and $w(t,x)=w_0(t,e\cdot x)$ on $U$ for some $w_0$.}\nonumber
\end{gather}
Then, we necessarily have $v=\bb u$ and the pair $(w,u)$ is the unique global weak solution of the following homogenized problem in~$\R^2$,
\begin{equation}\label{eq:homog-lake}
\left\{\begin{array}{l}
\partial_tw+\tfrac1{b_0}\Div(w\bb u)=g,\\
\Div(\bb u)=0,\\
\curl(u)=w,\\
w(t,x)|_{t=0}=w^\circ,
\end{array}\right.
\end{equation}
where $b_0:=\fint_Q b$ is the averaged depth and
where $\bb$ is the symmetric matrix given by
\begin{equation}\label{eq:def-barb-lake}
\bb\,:=\,J^T\bm{\bar a}^{-1}J
\,\in\,\R^{2\times2},\qquad J:=\begin{pmatrix}0&-1\\1&0\end{pmatrix},
\end{equation}
in terms of the homogenized matrix $\bm{\bar a}\in\R^{2\times2}$ defined componentwise by
\begin{equation}\label{eq:def-bara-lake}
\bm{\bar a}_{ij}\,:=\,\int_Q b^{-1}(\nabla\psi_i+e_i)\cdot(\nabla\psi_j+e_j)\,=\,e_i\cdot\int_Q b^{-1}(\nabla\psi_j+e_j),\qquad1\le i,j\le 2,
\end{equation}
where for $1\le i\le2$ the corrector $\psi_i\in H^1_\per(\R^2)$ is the unique mean-zero periodic solution of
\begin{equation}\label{eq:def-cor-lake}
\Div(b^{-1}(\nabla\psi_i+e_i))=0,\qquad\text{in $\R^2$}.
\end{equation}
\end{theor}

\begin{rems}
Comments are in order.
\begin{enumerate}[(a)]
\item \emph{Limit equations:}\\
Maybe surprisingly, since the homogenized matrix $\bm{\bar a}$ is not scalar in general, we find that the limit equations~\eqref{eq:homog-lake} do not take the form of lake equations with constant depth: indeed, the heterogeneous depth $b(\tfrac\cdot\e)$ is replaced by an effective (matrix) value~$\bb$ in the div-curl problem defining the limit velocity $u$ from the vorticity $w$, but the limit vorticity is transported by the incompressible field~$b_0^{-1}\bb u$ instead of the limit velocity~$u$ itself.
In fact, in terms of $v=\bb u$, we find that the pair $(w,v)$ satisfies an equation of the form~\eqref{eq:homog-imperm} with $1-\lambda$ and $\bm{\bar m}$ replaced by $b_0$ and $\bb^{-1}$, respectively,
\begin{equation*}
\left\{\begin{array}{l}
\partial_tw+\tfrac1{b_0}\Div(w v)=g,\\
\Div(v)=0,\\
\curl(\bb^{-1}v)=w,\\
w|_{t=0}=w^\circ.
\end{array}\right.
\end{equation*}

\item \emph{Dilute expansion:}\\
Let us consider the case of a depth function $b$ that is constant everywhere except in a periodic array of inclusions. More precisely, let $b$ be given by
\[b\,:=\,\left\{\begin{array}{ll}
\alpha,&\text{in $\R^d\setminus\bigcup_nI_n$,}\\
\beta,&\text{in $\bigcup_nI_n$,}
\end{array}\right.\]
for some given values $0<\alpha,\beta<\infty$, where the collection of inclusions $\{I_n\}_n$ is taken as in Theorem~\ref{th:imper}.
%\alpha \Id+\sum_{n} (\beta-\alpha)\Id \mathds1_{I_n}$ (with $I_n$ as in Theorem~\ref{th:imper}), and
In this setting, if inclusions are dilute, we can make the limit equation more explicit by appealing to an effective medium expansion in form of the Clausius--Mossotti formula:
under suitable dilution assumptions as described in~\cite{DG-CM}, denoting by $0\le\lambda<1$ the volume fraction of the inclusions, if each inclusion $I_n$ is the translation of a given solid~$I_0$, we find
\begin{equation}\label{eq:CM}
\bm{\bar a}\,=\,\alpha^{-1} \Id+\lambda\bm{\bar a}^0+o(\lambda),
\end{equation}
where $\bm{\bar a}^0$ is given by
\[\bm{\bar a}^0_{ij}\,:=\, (\beta^{-1}-\alpha^{-1}) e_i\cdot \fint_{I_0} (\nabla\psi_j^0+e_j),\]
in terms of the following single-inclusion corrector problem in $\R^2$,
\begin{equation*}
\Div( b^{-1} (\nabla \psi_j^0+e_j))=0.
\end{equation*}
In case of spherical inclusions, the single-inclusion corrector problem can be solved explicitly, which leads to
\[\bm{\bar a}^0\,=\,\frac{2\alpha^{-1} (\beta^{-1}-\alpha^{-1})}{\alpha^{-1}+\beta^{-1}}\Id\,=\,\frac{2(\alpha-\beta)}{\alpha(\alpha+\beta)}\Id,\]
cf.~\cite{DG-CM}, in which case we recover a standard lake equation to leading order in the homogenization limit.

\smallskip
\item \emph{Extensions:}\\
$\bullet$ We assume that the depth function $b$ is periodic, but the proof is easily adapted to a locally periodic setting. More precisely, this amounts to replacing $b(\frac x\e)$ by~$c(x,\frac x\epsilon)$ in~\eqref{eq:lake} provided that $c(x,\cdot)$ is periodic for all $x\in\R^2$. In that setting, we note that the homogenized matrix $\bm{\bar a}$ would further depend on the macroscopic variable $x$.
\\[1mm]
$\bullet$ We consider for simplicity the lake equations in the whole plane $\R^2$, but the proof is easily adapted to a bounded domain with suitable boundary conditions.
\\[1mm]
$\bullet$ A similar result can be obtained if periodicity is replaced by quasiperiodicity.
\end{enumerate}
\end{rems}

\subsection{Random setting and open problems}\label{sec:random}
We may wonder whether the above periodic homogenization results also hold in the random setting. An advantage of the random setting is that mixing ensures ergodicity in {\it every} direction, which would hopefully allow to avoid the non-degeneracy condition~\eqref{eq:non-deg-re} for limit points. Yet, as we briefly explain, our analysis leads us to a delicate probabilistic open question.
For shortness, we focus here on the homogenization of impermeable inclusions, but the argument is similar for lakes.
We first properly introduce the random setting:
\begin{enumerate}[---]
\item Let $\{I_n\}_n$ be a collection of random subsets of $\R^2$. More precisely, each $I_n$ is assumed to be a random Borel subset of $\R^2$, constructed on some underlying probability space~$(\Omega,\Pm)$: letting $\mathcal B(\R^2)$ denote the Borel $\sigma$-algebra on $\R^2$, this means that each $I_n$ is a map $\Omega\to\mathcal B(\R^2)$ such that~$\mathds1_{I_n}$ is measurable on the product $\R^2\times\Omega$ and such that $\mathds1_{I_n}(x)$ is measurable on~$\Omega$ for all $x\in\R^2$. We implicitly assume that $\Omega$ is endowed with the $\sigma$-algebra generated by $\{I_n\}_n$.
\smallskip\item We assume that the random set $\Ic=\bigcup_nI_n$ is stationary, which means that the finite-dimensional law of the translated random Borel set $x+\Ic$ does not depend on the shift~$x\in\R^2$.
\smallskip\item As~$\Omega$ is endowed with the $\sigma$-algebra generated by $\{I_n\}_n$, for a measurable function $g$ on~$\Omega$, we can define $T_zg$ as being obtained from $g$ by replacing the underlying $\{I_n\}_n$ by $\{-z+I_n\}_n$. By stationarity, $T_z$ defines an isometry on $\Ld^p(\Omega)$ for all $1\le p\le\infty$. A random field $h$ on $\R^2\times\Omega$ is said to be stationary if it is of the form $h(z,\cdot)=T_zg$ for some random variable $g$.
\smallskip\item We further assume that the random set $\Ic=\bigcup_nI_n$ is strongly mixing, which means that for any measurable set $A\subset\Omega$ we have $\E[\mathds1_AT_z\mathds1_A]\to\pr{A}^2$ as $|z|\uparrow\infty$, where $\E$ denotes the expectation.
\smallskip\item We assume that for some $\rho>0$ the random set $\Ic$ almost surely satisfies the same regularity and hardcore conditions as in Theorem~\ref{th:imper}.
\end{enumerate}
In this setting, for~{$1\le i\le2$}, we can define the corrector $\varphi_i\in H^1_\loc(\R^2;\Ld^2(\Omega))$ as the unique almost sure weak solution of~\eqref{eq:varphi-imperm} such that $\nabla\varphi_i$ is stationary, $\E[\nabla\varphi_i]=0$, and say anchoring~$\varphi_i(0)=0$; see e.g.~\cite[Section~8.6]{JKO94}. We shall use the notation $\varphi:=(\varphi_1,\varphi_2)$ and $\varphi_e:=e\cdot\varphi$. The homogenized matrix $\bm{\bar a}\in\R^{2\times2}$ is then defined componentwise by
\begin{equation}\label{eq:def-bara-re}
\bm{\bar a}_{ij}\,:=\,\delta_{ij}+\E[\nabla\varphi_i\cdot\nabla\varphi_j]\,=\,\E[(\nabla\varphi_i+e_i)\cdot(\nabla\varphi_j+e_j)],\qquad1\le i,j\le 2.
\end{equation}
In these terms, our analysis involves the following key question.

\begin{question}\label{qu:unique}
Given any $e\in\Sp^1$, if $\mu\in\Ld^2_\loc(\R^2;\Ld^2(\Omega))$ is a stationary random field that satisfies almost surely in the weak sense
\begin{equation}\label{eq:mu-invar}
\Div(\mu (e+\nabla\varphi_e)^\bot)=0,\qquad\text{in $\R^2$},
\end{equation}
then is it true that $\mu$ is necessarily almost surely constant in $\R^2\setminus\Ic$?
\end{question}

Uniqueness of the invariant measure associated with the corrector field is crucial for the homogenization of the vorticity transport equation.
As the following shows, the answer to this question is positive provided that correctors have bounded second moments. Yet, correctors are generically unbounded in 2D even under the strongest mixing assumptions.

\begin{lem}\label{lem:unique}
As above, let $\{I_n\}_n$ be a random collection of subsets of $\R^2$ such that the random set $\Ic=\bigcup_nI_n$ is stationary and strongly mixing, and such that for some $\rho>0$ it almost surely satisfies the same regularity and hardcore conditions as in Theorem~\ref{th:imper}.
Then, the answer to Question~\ref{qu:unique} is positive if $\sup_x\E[|\varphi(x)|^2]<\infty$.
\end{lem}

\begin{proof}
Let $e\in\Sp^1$ and let $\mu\in\Ld^2_\loc(\R^2;\Ld^2(\Omega))$ be a stationary random field that satisfies equation~\eqref{eq:mu-invar} almost surely in the weak sense.
If the corrector $\varphi$ satisfies the assumption $\sup_x\E[|\varphi(x)|^2]<\infty$, then there  exists a {\it stationary} random field $\tilde\varphi\in H^1_\loc(\R^2;\Ld^2(\Omega)^2)$ such that $\nabla\tilde\varphi=\nabla\varphi$; this can be checked e.g.\@ by means of a massive approximation argument (as standard in elliptic homogenization). Let $\tilde\varphi_e:=e\cdot\tilde\varphi$.
Almost surely, as $\mu$ satisfies equation~\eqref{eq:mu-invar} in the weak sense, we deduce that it is constant along characteristics of the vector field $(e+\nabla\varphi_e)^\bot=(e+\nabla\tilde\varphi_e)^\bot$, hence we have
\begin{equation}\label{eq:mu-redmu0}
\mu(x)\,=\,\mu_0(x\cdot e+\tilde\varphi_e(x)),\qquad\text{for all $x\in\R^2\setminus\Ic$,}
\end{equation}
for some measurable map $\mu_0$ on $\R\times\Omega$. 

We first argue that the map $\R^2 \setminus \Ic \to\R:x\mapsto x\cdot e+\tilde\varphi_e(x)$ is almost surely surjective. By Theorem~\ref{th:diffeo-rigid}, the map $\Phi:\R^2\setminus\Ic \to \R^2:x \mapsto x+\tilde\varphi(x)$ is almost surely a diffeomorphism onto its image $\R^2 \setminus \Phi(\Ic)$. In particular, we find that almost surely $\Phi(\Ic)=\bigcup_n \Phi(I_n)$ is the union of disjoint bounded open sets in $\R^2$, so that for all $b\in \Sp^1$ and $r\in \R$ the line $L_{b,r}=\{x \in \R^2:x\cdot b=r\}$ satisfies $\mathcal{H}^1(L_{b,r} \setminus \Phi(\Ic))>0$ (where $\mathcal H^1$ stands for the Hausdorff measure), and thus also $\mathcal{H}^1( \Phi^{-1} (L_{b,r} \setminus \Phi(\Ic)))>0$. Applied to $b=e$ and to any $r\in \R$, this implies that the map $\R^2 \setminus \Ic \to\R:x\mapsto x\cdot e+\tilde\varphi_e(x)$ is almost surely surjective.

With this observation at hand, let us now return to the analysis of $\mu$.
For a stationary random field $h$, note that the definition of the isometries $\{T_z\}_{z\in\R^2}$ on the probability space entails $T_zh(x)=T_zT_xh=h(x+z)$.
Given $z\in\R^2$, recalling~\eqref{eq:mu-redmu0} and using the stationarity of $\mu$, we then find almost surely for $x\in\R^2\setminus(-z+\Ic)$,
\begin{eqnarray*}
(T_z\mu_0)(x\cdot e+(T_z\tilde \varphi_e)(x))
&=&(T_z\mu)(x)
\\
&=&\mu(x+z)
\\
&=& \mu_0(x\cdot e+z\cdot e+\tilde\varphi_e(x+z)).
\end{eqnarray*}
By the fundamental theorem of calculus, the stationarity of $\nabla\tilde\varphi_e=\nabla \varphi_e$ yields
\begin{eqnarray*}
(T_z \tilde\varphi_e)(x)-(T_z \tilde \varphi_e)(0)
&=&\int_0^1x\cdot\nabla(T_z \tilde\varphi_e)(tx)\,dt\\
&=&\int_0^1x\cdot(\nabla\tilde\varphi_e)(tx+z)\,dt\\
&=&\tilde\varphi_e(x+z)-\tilde\varphi_e(z),
\end{eqnarray*}
so the above becomes
\begin{equation*}
(T_z\mu_0)(x\cdot e+(T_z\tilde \varphi_e)(x))
\,=\, \mu_0(x\cdot e+(T_z \tilde \varphi_e)(x)+z\cdot e+ \tilde \varphi_e(z)-(T_z \tilde \varphi_e)(0)).
\end{equation*}
By the almost sure surjectivity of the map $\R^2 \setminus(-z+\Ic) \to\R:x\mapsto x\cdot e+(T_z\tilde\varphi_e)(x)$, there exists $x\in \R^2 \setminus(-z+\Ic)$ such that $x\cdot e+(T_z\tilde \varphi_e)(x)=0$, and therefore we obtain almost surely for all~$z\in \R^2$,
\begin{equation}\label{eq:mu0-stat}
(T_z\mu_0)(0)\,=\,\mu_0(z\cdot e+ \tilde \varphi_e(z)-(T_z \tilde \varphi_e)(0))\,=\,\mu_0(z\cdot e),
\end{equation}
where the last equality follows from the stationarity of $\tilde \varphi_e$.
In particular, this entails that the random variable $\mu_0(0)$ is invariant under $\{T_{se^\bot}\}_{s\in\R}$. As the strong mixing assumption implies directional ergodicity, we can conclude $\mu_0(0)=\bar \mu_0(0):=\expec{\mu_0(0)}$ almost surely. Similarly, $\mu_0(r)=\bar \mu_0(r):=\expec{\mu_0(r)}$ almost surely for all $r\in\R$.
Now, recalling~\eqref{eq:mu-redmu0} and the stationarity of $\mu$ and $\tilde\varphi_e$, we infer almost surely, for all $x\in\R^2\setminus\Ic$ and $z\in\R^2$,
\begin{eqnarray*}
\mu(x)&=&(T_{-z}\mu)(z+x)\\
&=&\bar \mu_0(z\cdot e+x\cdot e+(T_{-z}\tilde \varphi_e)(x+z))\\
&=&\bar \mu_0(z\cdot e+x\cdot e+\tilde \varphi_e(x)).
\end{eqnarray*}
By the arbitrariness of $z\in\R^2$, this implies that $\bar \mu_0$ is a constant function on $\R$, and hence~$\mu$ is a constant function on $\R^2\setminus\Ic$.
\end{proof}
\begin{rem}
Following the above proof, if we do not know whether a stationary corrector exists, we can still deduce that any almost sure stationary weak solution $\mu\in\Ld^2_\loc(\R^2;\Ld^2(\Omega))$ of equation~\eqref{eq:mu-invar} can be written as
\[\mu(x)=\mu_0(x\cdot e+\varphi_e(x)),\]
for some measurable map $\mu_0$ on $\R\times\Omega$, where $\varphi_e$ is the non-stationary corrector with anchoring $\varphi_e(0)=0$.
As corrector increments are always stationary (since $\nabla\varphi_e$ is stationary), we can deduce that $\mu_0$ satisfies almost surely, instead of~\eqref{eq:mu0-stat},
\[(T_z\mu_0)(0)\,=\,\mu_0(z\cdot e+\varphi_e(z)),\qquad\text{for all $z\in\R^2$.}\]
Yet, it is unclear to us how to exploit the resulting invariance of $\mu_0$ upon shifts orthogonally to the random curves $\{z\in\R^2:z\cdot e+\varphi_e(z)=r\}$, $r\in\R$.
\end{rem}

We do not know whether there are statistical assumptions that guarantee that the answer to Question~\ref{qu:unique} is  positive.
Assuming a positive answer to this question, the rest of our analysis still holds in the random setting and  leads to the following homogenization result.

\begin{theor}\label{th:imper-rand}
As above, let $\{I_n\}_n$ be a random collection of subsets of $\R^2$ such that the random set $\Ic=\bigcup_nI_n$ is stationary and strongly mixing, and such that for some $\rho>0$ it almost surely satisfies the same regularity and hardcore conditions as in Theorem~\ref{th:imper}.
For all $\e>0$, let $u_\e$ be the unique solution of the 2D Euler system~\eqref{eq:imperm-main} with $\e$-rescaled impermeable inclusions at $\Ic_\e=\e\Ic=\bigcup_n\e I_n$, and let the data of the problem converge as~$\e\downarrow0$ in the sense of~\eqref{eq:conv-data-imperm}.
Provided that the answer to Question~\ref{qu:unique} is positive, we have almost surely
\begin{equation}\label{eq:lim-omu-res}
\begin{array}{rcll}
\curl(u_\e)\mathds1_{\R^2\setminus\Ic_\e}&\cvfs&w,&\text{in $\Ld^\infty_\loc(\R_+;\Ld^1\cap\Ld^\infty(\R^2))$},\\
u_\e\mathds1_{\R^2\setminus\Ic_\e}&\cvfs& u,&\text{in $\Ld^\infty_\loc(\R_+;\Ld^2(\R^2))$},
\end{array}\end{equation}
where the limit $(w,u)$ is the unique global weak solution of the homogenized problem~\eqref{eq:homog-imperm},
where we now denote by $\lambda:=\E[\mathds1_\Ic]$ the inclusions' volume fraction
and where the coefficient~$\bm{\bar m}$ is given by~\eqref{eq:def-barm} in terms of the homogenized matrix $\bm{\bar a}\in\R^{2\times2}$ now defined in~\eqref{eq:def-bara-re}.
\end{theor}

\section{Homogenization of impermeable inclusions}\label{sec:homog-imper}
This section is devoted to the proof of Theorems~\ref{th:imper} and~\ref{th:imper-rand}.
Since the periodic setting can be viewed as a particular case of the stationary random setting of Section~\ref{sec:random} with $\Omega=Q=[-\frac12,\frac12)^2$ the periodicity cell, we focus on the latter. Let $\{I_n\}_n$ be a collection of random subsets of~$\R^2$ such that the random set $\Ic=\bigcup_nI_n$ is stationary and strongly mixing, and such that for some $\rho>0$ it almost surely satisfies the same regularity and hardcore conditions as in Theorem~\ref{th:imper}.
With the assumptions on the data $u_\e^\circ,f_\e$, cf.~\eqref{eq:conv-data-imperm}, almost surely, the 2D Euler system~\eqref{eq:imperm-main} admits a unique global weak solution $u_\e\in\Ld^\infty_{\loc}(\R_+;\Ld^2(\R^2\setminus\Ic_\e)^2)$ with vorticity $w_\e:=\curl(u_\e)\in\Ld^\infty_{\loc}(\R_+;\Ld^1\cap\Ld^\infty(\R^2\setminus\Ic_\e))$, cf.~\cite{Yudovich-63,Kikuchi,Majda-Bertozzi}. In addition, in terms of $w^\circ_\e:=\curl(u_\e^\circ)$ and $g_\e:=\curl(f_\e)$, the Euler equations may alternatively be written in vorticity formulation,
\begin{equation}\label{eq:Euler-vort}
\left\{\begin{array}{ll}
\partial_tw_\e+\Div(u_\e w_\e)=g_\e,&\text{in $\R^2\setminus \Ic_\e$},\\
\Div(u_\e)=0,&\text{in $\R^2\setminus\Ic_\e$},\\
\curl(u_\e)=w_\e,&\text{in $\R^2\setminus\Ic_\e$},\\
u_\e\cdot\nu=0,&\text{on $\partial \Ic_\e$},\\
\int_{\e\partial I_n}u_\e\cdot\tau=0,&\forall n,\\
w_\e|_{t=0}=w^\circ_\e,&
\end{array}\right.
\end{equation}
and we have the following a priori estimates, for all $t\ge0$,
\begin{equation}\label{eq:bnd-omueps-apriori}
\begin{array}{rcl}
\|w_\e(t)\|_{\Ld^1\cap\Ld^\infty(\R^2\setminus\Ic_\e)}&\le&\|w_\e^\circ\|_{\Ld^1\cap\Ld^\infty(\R^2\setminus\Ic_\e)}+\|g_\e\|_{\Ld^1(0,t;\Ld^1\cap\Ld^\infty(\R^2\setminus\Ic_\e))},\\[2mm]
\|u_\e(t)\|_{\Ld^2(\R^2\setminus\Ic_\e)}&\le&\|u_\e^\circ\|_{\Ld^2(\R^2\setminus\Ic_\e)}+\|f_\e\|_{\Ld^1(0,t;\Ld^2(\R^2\setminus\Ic_\e))},
\end{array}
\end{equation}
where the right-hand sides are uniformly bounded as $\e\downarrow0$ by assumption~\eqref{eq:conv-data-imperm}.
From the div-curl problem and the impermeability condition in~\eqref{eq:Euler-vort}, it is easily checked that $w_\e\mathds1_{\R^2\setminus\Ic_\e}$ is also bounded in $\Ld^\infty_\loc(\R_+;\dot H^{-1}(\R^2))$ (e.g.\@ using the trace estimate of Lemma~\ref{lem:trace} below).
By weak compactness, in terms of the extensions
\[\tilde w_\e\,:=\,w_\e\mathds1_{\R^2\setminus\Ic_\e},\qquad\tilde u_\e\,:=\,u_\e\mathds1_{\R^2\setminus\Ic_\e},\]
we deduce almost surely, up to a subsequence,
\begin{equation}\label{eq:conv-omu-extr}
\begin{array}{rcll}
\tilde w_\e&\cvfs&w,&\text{in $\Ld^\infty_\loc(\R_+;\Ld^1\cap\Ld^\infty\cap\dot H^{-1}(\R^2))$},\\
\tilde u_\e&\cvfs& u,&\text{in $\Ld^\infty_\loc(\R_+;\Ld^2(\R^2))$},
\end{array}\end{equation}
for some limit $(w,u)$, which remains to be characterized.
We split the proof into two main steps, first analyzing the div-curl problem defining $\tilde u_\e$, characterizing its limit $u$ and the oscillatory behavior of $\tilde u_\e$, before turning to the homogenization of the transport equation for the vorticity $\tilde w_\e$.

\subsection{Homogenization of the div-curl problem}
We start with the homogenization of the underlying div-curl problem for $u_\e$ in~\eqref{eq:Euler-vort}, which is equivalent to a scalar elliptic problem with stiff inclusions for the stream function (see~\eqref{eq:divcurl} below). Homogenization for this type of elliptic problem is classical, see e.g.~\cite[Section~8.6]{JKO94}. By means of a corrector-type result, we further manage to describe the oscillations of the fluid velocity~$u_\e$ in a strong $\Ld^2$ sense.
%As we could not find a reference for corrector results in case of stiff inclusions, a proof is included below.
Note that the oscillations of $u_\e$ are independent of those of the underlying vorticity~$w_\e$ as a consequence of the compactness of the map~$w_\e\mapsto u_\e$ on~$\Ld^2(\R^2\setminus\Ic_\e)$.

\begin{prop}\label{prop:omu-lim}
Almost surely, the extracted limit point $(w,u)$ in~\eqref{eq:conv-omu-extr} satisfies the following properties.
\begin{enumerate}[(i)]
\item \emph{Homogenized div-curl problem:}
\begin{equation}\label{eq:lim-divcurl}
\left\{\begin{array}{rcll}
\Div(u)&=&0,&\text{in $\R^2$},\\[1mm]
\curl(\bm{\bar m} u)&=&w,&\text{in $\R^2$},
\end{array}\right.\end{equation}
where $\bm{\bar m}\in\R^{2\times 2}$ is the symmetric matrix defined in~\eqref{eq:def-barm} and~\eqref{eq:def-bara-re}. In particular, this implies additional regularity: the limit $u$ necessarily belongs to $\Ld^\infty_\loc(\R_+;W^{1,p}\cap W^{r,\infty}(\R^2)^2)$ for all $2\le p<\infty$ and $r<1$.
\smallskip\item \emph{Corrector result:} along the extracted subsequence, we have the strong convergence
\begin{equation}\label{eq:corr-res-veps}
\tilde u_\e+(e_i+\nabla\varphi_i)^\bot(\tfrac\cdot\e)\,u^\bot_i~\to~0,\qquad\text{in $\Ld^\infty_\loc(\R_+;\Ld^2_{\loc}(\R^2))$,}
\end{equation}
where for all $1\le i\le2$ the corrector $\varphi_i\in H^1_\loc(\R^2;\Ld^2(\Omega))$ is defined in~\eqref{eq:def-bara-re}.
\end{enumerate}
\end{prop}

\begin{proof}
Due to the impermeability condition $u_\e\cdot\nu=0$ on $\partial\Ic_\e$, we find that the extended fluid velocity $\tilde u_\e$ actually satisfies
\[\Div(\tilde u_\e)=0,\qquad\text{in $\R^2$}.\]
Hence, there exists a stream function $s_\e\in\Ld^\infty_\loc(\R_+;\dot H^1(\R^2))$ such that
\begin{equation}\label{eq:rel-u/s}
\tilde u_\e\,=\,\nabla^\bot s_\e,
\end{equation}
and the div-curl problem for $u_\e$ in~\eqref{eq:Euler-vort} then takes form of the following elliptic problem for~$s_\e$,
\begin{equation}\label{eq:divcurl}
\left\{\begin{array}{ll}
\triangle s_\e=w_\e,&\text{in $\R^2\setminus\Ic_\e$},\\
\nabla s_\e=0,&\text{in $\Ic_\e$},\\
\int_{\e\partial I_n}\partial_\nu s_\e=0,&\forall n.
\end{array}\right.
\end{equation}
In other words, the stream function $s_\e$ satisfies an elliptic problem with so-called stiff inclusions at~$\Ic_\e$.
From there, we split the proof into three steps.

\medskip
\step1 Proof of~(i): qualitative homogenization.\\
By the standard qualitative theory of stochastic homogenization for stiff inclusions, e.g.~\cite[Section~8.6]{JKO94}, almost surely, given the weak $\Ld^2$ convergence~\eqref{eq:conv-omu-extr} of the vorticity, we have for the solution $s_\e$ of~\eqref{eq:divcurl},
\begin{equation}\label{eq:qual-homog-stiff}
\nabla s_\e\cvfs\nabla s,\qquad\text{in~$\Ld^\infty_\loc(\R_+;\Ld^2(\R^2))$,}
\end{equation}
where the limit is the unique solution of
\begin{equation}\label{eq:homog-eqn}
\Div(\bm{\bar a}\nabla s)=w,\qquad\text{in $\R^2$},
\end{equation}
where $\bm{\bar a}$ is the homogenized matrix defined in~\eqref{eq:def-bara-re}.
Recalling~\eqref{eq:conv-omu-extr}, \eqref{eq:rel-u/s}, as well as the relation~\eqref{eq:def-barm} between $\bm{\bar a}$ and $\bm{\bar m}$, we deduce that $u=\nabla^\bot s$ satisfies the homogenized div-curl problem~\eqref{eq:lim-divcurl}.
Finally, as the limit vorticity $w$ belongs to $\Ld^\infty_\loc(\R_+;\Ld^1\cap\Ld^\infty\cap\dot H^{-1}(\R^2))$, the stated additional regularity for $u=\nabla^\bot s$ follows from equation~\eqref{eq:homog-eqn} using the Riesz and Calder\'on--Zygmund theories.
%For later purposes, we also note that $s=\triangle^{-1}\curl(u)\in \Ld^\infty_\loc(\R_+;\Ld^p(\R^2))$ for all $2<p<\infty$.

\medskip
\step2 Time compactness.\\
% $\omega_\e\mathds1_{\R^2\setminus\Ic_\e}\to\omega$ strongly in $C(\R_+;H^{-\eta}_\loc(\R^2))$ for all $\eta>0$.\\
Before turning to the proof of the corrector result, some care is needed due to time compactness issues that prevent the direct application of standard homogenization techniques. Consider the following extensions of the data,
\[\tilde w_\e^\circ=w_\e^\circ\mathds1_{\R^2\setminus\Ic_\e},\qquad\tilde g_\e:=g_\e\mathds1_{\R^2\setminus\Ic_\e},\qquad \tilde f_\e:=f_\e\mathds1_{\R^2\setminus\Ic_\e},\]
and note that the assumption $f_\e\cdot\tau=0$ on $\partial\Ic_\e$ ensures
\begin{equation}\label{eq:link-fg}
\tilde g_\e=\curl(\tilde f_\e),\qquad\text{in $\R^2$}.
\end{equation}
In these terms, due to the impermeability condition $u_\e\cdot\nu=0$ on $\partial\Ic_\e$, we note that the transport equation for the vorticity in~\eqref{eq:Euler-vort} actually takes on the following form for the extended vorticity $\tilde w_\e$,
\begin{equation}\label{eq:vort-tildeom-tsp}
%\left\{\begin{array}{ll}
\partial_t\tilde w_\e+\Div(\tilde u_\e\tilde w_\e)=\tilde g_\e,\qquad\text{in $\R^2$}.
%\tilde w_\e|_{t=0}=\tilde w_\e^\circ.
%\end{array}\right.
\end{equation}
Using the boundedness properties for $\tilde w_\e$ and~$\tilde u_\e$, cf.~\eqref{eq:bnd-omueps-apriori}, the weak formulation for this equation entails that $\{\partial_t\tilde w_\e\}_\e$ is bounded in $\Ld^\infty_\loc(\R_+; H^{-1}(\R^2))$.
As in addition $\{\tilde w_\e\}_\e$ is bounded in $\Ld^\infty_\loc(\R_+;\Ld^1\cap\Ld^\infty(\R^2))$, we deduce from the Aubin--Lions lemma that it is precompact in $C_\loc(\R_+;H^{-1}_\loc(\R^2))$. Hence, we get along the extracted subsequence,
\begin{equation}\label{eq:strongconv-om-imp}
\tilde w_\e\,\to\, w,\qquad\text{in $C_\loc(\R_+;H^{-1}_\loc(\R^2))$}.
\end{equation}
Next, from equation~\eqref{eq:divcurl}, after taking the time derivative, we find
\[\int_{\R^2}|\nabla\partial_ts_\e|^2\,=\,-\int_{\R^2}(\partial_ts_\e)(\partial_t\tilde w_\e),\]
which entails, using~\eqref{eq:vort-tildeom-tsp} and~\eqref{eq:link-fg},
\[\|\nabla\partial_ts_\e\|_{\Ld^2(\R^2)}\,\le\,\|f_\e\|_{\Ld^2(\R^2)}+\|\tilde u_\e\|_{\Ld^2(\R^2)}\|\tilde w_\e\|_{\Ld^\infty(\R^2)},\]
thus proving that $\{\partial_ts_\e\}_\e$ is bounded in $\Ld^\infty_\loc(\R_+;\dot H^1(\R^2))$. From the weak convergence~\eqref{eq:qual-homog-stiff} combined  to the Poincar\'e inequality, the Rellich theorem, and the Aubin--Simon lemma, we similarly deduce along the extracted subsequence, for all $R>0$,
\begin{equation}\label{eq:strongconv-s-imper}
s_\e-\fint_{B_R}s_\e\,\to\,S_R,\qquad\text{in $C_\loc(\R_+;\Ld^2(B_R))$},
\end{equation}
for some limit $S_R$ with
\begin{equation}\label{eq:nabSRs}
\nabla S_R=\nabla s,\qquad\text{in $B_R$.}
\end{equation}
Noting that the regularity in Step~1 entails $s\in\Ld^\infty_\loc(\R_+;\dot H^{\delta}(\R^2))$ for all $0<\delta\le1$, we find
\[\|S_R\|_{\Ld^2(B_R)}\,\lesssim\,R^\delta\|\nabla^\delta S_R\|_{\Ld^2(B_R)}\,\le\,R^\delta\|\nabla^\delta s\|_{\Ld^2(\R^2)},\]
and thus
\begin{equation}\label{eq:subliner-SR}
\lim_{R\uparrow\infty}R^{-1}\|S_R\|_{\Ld^2(B_R)}\,=\,0.
\end{equation}

\medskip
\step3 Proof of~(ii): corrector result.\\
Given $T>0$, for all $R>0$, consider a cut-off function $\chi_R\in C^\infty_c(\R_+\times\R^2)$ supported in $[0,T]\times B_R$ with $\|\nabla\chi_R\|_{\Ld^\infty(\R_+\times\R^2)}\lesssim R^{-1}$.
We can decompose
\begin{multline}
\int_{\R_+\times\R^2}\chi_R\big|\nabla s_\e-(e_i+\nabla\varphi_i)(\tfrac\cdot\e)\nabla_is\big|^2\\
\,=\,\int_{\R_+\times\R^2}\chi_R|\nabla s_\e|^2-2\int_{\R_+\times\R^2}\chi_R(\nabla_is)\nabla s_\e\cdot (e_i+\nabla\varphi_i)(\tfrac\cdot\e)\\
+\int_{\R_+\times\R^2}\chi_R(\nabla_is)(\nabla_js)\big((e_i+\nabla\varphi_i)\cdot(e_j+\nabla\varphi_j)\big)(\tfrac\cdot\e).
\label{eq:decomp-corre-approx}
\end{multline}
We shall pass to the limit separately in each right-hand side term by means of compensated compactness. We split the proof into four further substeps.

\medskip
\substep{3.1} Proof that
\begin{equation}\label{eq:lim-first-term-res}
\lim_{R\uparrow\infty}\limsup_{\e\downarrow0}\bigg|\int_{\R_+\times\R^2}\chi_R|\nabla s_\e|^2-\int_{\R_+\times\R^2}\nabla s\cdot \bm{\bar a}\nabla s\bigg|\,=\,0.
\end{equation}
Integrating by parts and using~\eqref{eq:divcurl}, we find
\begin{multline*}
\int_{\R_+\times\R^2}\chi_R|\nabla s_\e|^2\,=\,-\int_{\R_+\times\R^2}\Big(s_\e-\fint_{B_R}s_\e\Big)\nabla\chi_R\cdot\nabla s_\e-\int_{\R_+\times\R^2}\Big(s_\e-\fint_{B_R}s_\e\Big)\chi_R\tilde w_\e\\
-\sum_n\int_{\R_+}\bigg(\fint_{\e I_n}\Big(s_\e-\fint_{B_R}s_\e\Big)\bigg) \int_{\e\partial I_n}\chi_R \partial_\nu s_\e.
\end{multline*}
Using~\eqref{eq:conv-omu-extr}, \eqref{eq:qual-homog-stiff}, and~\eqref{eq:strongconv-s-imper} to pass to the limit in the first two right-hand side terms, and using the trace estimate of Lemma~\ref{lem:trace} to estimate the last term, we get
\begin{multline*}
\limsup_{\e\downarrow0}\bigg|\int_{\R_+\times\R^2}\chi_R|\nabla s_\e|^2+\int_{\R_+\times\R^2}S_R\nabla\chi_R\cdot\nabla s+\int_{\R_+\times\R^2}\chi_RS_R w\bigg|\\
\,\lesssim\,\|\nabla\chi_R\|_{\Ld^\infty(\R^2)}\int_{0}^T\|S_R\|_{\Ld^2(B_{R})},
\end{multline*}
where we used that $\nabla s_\e$ and $\tilde w_\e$ are bounded in $\Ld^\infty_\loc(\R_+;\Ld^2(\R^2))$.
Rewriting the third left-hand side term using the homogenized equation~\eqref{eq:homog-eqn} and  $\nabla S_R=\nabla s$, we have
\[
\int_{\R_+\times\R^2}\chi_RS_R w\,=\,-\int_{\R_+\times\R^2}\chi_R\nabla s \cdot \bar a \nabla s-\int_{\R_+\times\R^2}S_R \nabla \chi_R \cdot \bar a \nabla s,
\]
this entails
\begin{equation*}
\limsup_{\e\downarrow0}\bigg|\int_{\R_+\times\R^2}\chi_R|\nabla s_\e|^2-\int_{\R_+\times\R^2}\chi_R\nabla s\cdot \bm{\bar a}\nabla s\bigg|
\,\lesssim\,\|\nabla\chi_R\|_{\Ld^\infty(\R^2)}\int_{0}^T\|S_R\|_{\Ld^2(B_{R})}.
\end{equation*}
By~\eqref{eq:subliner-SR}, the right-hand side tends to $0$ as $R\uparrow\infty$ and the claim~\eqref{eq:lim-first-term-res} follows.

\medskip
\substep{3.2} Proof that
\begin{equation}\label{eq:lim-cross-term-res}
\lim_{\e\downarrow0}\int_{\R_+\times\R^2}\chi_R(\nabla_is)\nabla s_\e\cdot (e_i+\nabla\varphi_i)(\tfrac\cdot\e)
\,=\,\int_{\R_+\times\R^2}\chi_R\nabla s\cdot\bm{\bar a}\nabla s.
\end{equation}
We start with a regularization argument for $s$: given $\eta>0$, let $s^\eta\in \Ld^\infty_\loc(\R_+;C^\infty_b(\R^2))$ with
\begin{equation}\label{eq:approx-s-seta-00}
\|\nabla s-\nabla s^\eta\|_{\Ld^\infty(0,T;\Ld^2(\R^2))}\le\eta.
\end{equation}
We may then estimate
\begin{multline*}
\bigg|\int_{\R_+\times\R^2}\chi_R(\nabla_is)\nabla s_\e\cdot (e_i+\nabla\varphi_i)(\tfrac\cdot\e)-\int_{\R_+\times\R^2}\chi_R(\nabla_is^\eta)\nabla s_\e\cdot (e_i+\nabla\varphi_i)(\tfrac\cdot\e)\bigg|\\
\,\le\,\|\nabla s_\e\|_{\Ld^\infty(0,T;\Ld^2(\R^2))}\|(e_i+\nabla\varphi_i)(\tfrac\cdot\e)(\nabla_is-\nabla_is^\eta)\|_{\Ld^\infty(0,T;\Ld^2(\R^2))},
\end{multline*}
and thus, by~\eqref{eq:qual-homog-stiff}, the stationarity of $\nabla\varphi$, and the ergodic theorem,
\begin{multline}\label{eq:approx-s-seta}
\limsup_{\e\downarrow0}\bigg|\int_{\R_+\times\R^2}\chi_R(\nabla_is)\nabla s_\e\cdot (e_i+\nabla\varphi_i)(\tfrac\cdot\e)-\int_{\R_+\times\R^2}\chi_R(\nabla_is^\eta)\nabla s_\e\cdot (e_i+\nabla\varphi_i)(\tfrac\cdot\e)\bigg|\\
\,\lesssim\,\|\nabla s-\nabla s^\eta\|_{\Ld^\infty(0,T;\Ld^2(\R^2))}\,\le\,\eta.
\end{multline}
We now focus on the term of interest with $s$ replaced by $s^\eta$.
Integrating by parts, using the corrector equation~\eqref{eq:varphi-imperm}, and the stiff boundary conditions for $s_\e$, we find
\begin{multline}\label{eq:mixed-term-cor}
\int_{\R_+\times\R^2}\chi_R(\nabla_is^\eta)\nabla s_\e\cdot (e_i+\nabla\varphi_i)(\tfrac\cdot\e)\\
\,=\,-\int_{\R_+\times\R^2}\Big(s_\e-\fint_{B_R}s_\e\Big)\nabla(\chi_R\nabla_is^\eta) \cdot (e_i+\nabla\varphi_i)(\tfrac\cdot\e)\\
-\sum_n\int_{\R_+}\bigg(\fint_{\e I_n}\Big(s_\e-\fint_{B_R}s_\e\Big)\bigg)\\
\times\bigg(\int_{\e\partial I_n}\Big(\chi_R\nabla_is^\eta-\fint_{\e I_n}\chi_R\nabla_is^\eta\Big) (e_i+\nabla\varphi_i)(\tfrac\cdot\e)\cdot\nu\bigg).
\end{multline}
As $s^\eta$ is smooth, we can approximate
\[\bigg\|\Big(\chi_R\nabla_is^\eta-\fint_{\e I_n}\chi_R\nabla_is^\eta\Big)-(\cdot-\e x_n)_j\Big(\fint_{\e I_n}\nabla_j(\chi_R\nabla_is^\eta)\Big)\bigg\|_{\Ld^\infty(0,T;W^{1,\infty}(\e I_n))}\,\le\,C_\eta \e,\]
where we have set $x_n:=\fint_{I_n}x\,dx$.
Using this to reformulate the last right-hand term of~\eqref{eq:mixed-term-cor}, and using the trace estimate of Lemma~\ref{lem:trace} to estimate the error, we get
\begin{multline*}
\bigg|\int_{\R_+\times\R^2}\chi_R(\nabla_is^\eta)\nabla s_\e\cdot (e_i+\nabla\varphi_i)(\tfrac\cdot\e)\\
+\int_{\R_+\times\R^2}\Big(s_\e-\fint_{B_R}s_\e\Big)\nabla(\chi_R\nabla_is^\eta) \cdot (e_i+\nabla\varphi_i)(\tfrac\cdot\e)\\
+\int_{\R_+\times\R^2}\Big(s_\e-\fint_{B_R}s_\e\Big)\nabla_j(\chi_R\nabla_is^\eta)\bigg(\sum_n\frac{\mathds1_{I_n}(\tfrac\cdot\e)}{|I_n|}\int_{\partial I_n}(\cdot- x_n)_j(e_i+\nabla\varphi_i)\cdot\nu\bigg)\bigg|\\
\,\lesssim\,C_\eta\e\sum_n\int_{0}^T\Big\|s_\e-\fint_{B_R}s_\e\Big\|_{\Ld^2(B_{R+1})}\|e_i+\nabla\varphi_i(\tfrac\cdot\e)\|_{\Ld^2(B_{R+1})}.
\end{multline*}
Using~\eqref{eq:qual-homog-stiff}, \eqref{eq:strongconv-s-imper}, and the ergodic theorem to pass to the limit in the different terms, we are led to
\begin{equation}\label{eq:lim-cross-term-corr}
\lim_{\e\downarrow0}\int_{\R_+\times\R^2}\chi_R(\nabla_is^\eta)\nabla s_\e\cdot (e_i+\nabla\varphi_i)(\tfrac\cdot\e)
\,=\,-\int_{\R_+\times\R^2}S_R\,\Div(\chi_R\bm{\bar c}\nabla s^\eta),
\end{equation}
where we have defined
\[\bm{\bar c}_{ij}\,:=\,\delta_{ij}+\E\bigg[\sum_n\frac{\mathds1_{I_n}}{|I_n|}\int_{\partial I_n}(\cdot- x_n)_i(e_j+\nabla\varphi_j)\cdot\nu\bigg].\]
Using the corrector equations~\eqref{eq:varphi-imperm}, a direct computation yields
\begin{eqnarray*}
\bm{\bar c}_{ij}&=&\delta_{ij}+\delta_{ij}\E[\mathds1_\Ic]-\E\bigg[\sum_n\frac{\mathds1_{I_n}}{|I_n|}\int_{\partial I_n}\varphi_i\partial_\nu\varphi_j\bigg]\\
&=&\delta_{ij}+\delta_{ij}\E[\mathds1_\Ic]+\E\big[\mathds1_{\R^2\setminus\Ic}\nabla\varphi_i\cdot\nabla\varphi_j\big]\\
&=&\delta_{ij}+\E\big[\nabla\varphi_i\cdot\nabla\varphi_j\big],
\end{eqnarray*}
that is, by definition~\eqref{eq:def-bara-re}, $\bm{\bar c}=\bm{\bar a}$.
Inserting this into~\eqref{eq:lim-cross-term-corr} and integrating by parts, we get
\begin{equation*}
\lim_{\e\downarrow0}\int_{\R_+\times\R^2}\chi_R(\nabla_is^\eta)\nabla s_\e\cdot (e_i+\nabla\varphi_i)(\tfrac\cdot\e)
\,=\,\int_{\R_+\times\R^2}\chi_R\nabla s\cdot\bm{\bar a}\nabla s^\eta.
\end{equation*}
Using~\eqref{eq:approx-s-seta-00} and~\eqref{eq:approx-s-seta} to get rid of the approximation, the claim~\eqref{eq:lim-cross-term-res} follows.

\medskip
\substep{3.3} Conclusion.\\
By the ergodic theorem, recalling~\eqref{eq:def-bara-re} and $\nabla s\in\Ld^\infty_\loc(\R_+;\Ld^2\cap\Ld^\infty(\R^2))$, we find almost surely
\[\lim_{\e\downarrow0}\int_{\R_+\times\R^2}\chi_R(\nabla_is)(\nabla_js)\big((e_i+\nabla\varphi_i)\cdot(e_j+\nabla\varphi_j)\big)(\tfrac\cdot\e)\,=\,\int_{\R_+\times\R^2}\chi_R\nabla s\cdot\bm{\bar a}\nabla s.\]
Inserting this together with~\eqref{eq:lim-first-term-res} and~\eqref{eq:lim-cross-term-res} into~\eqref{eq:decomp-corre-approx}, we deduce
\[\lim_{R\uparrow\infty}\lim_{\e\downarrow0}\int_{\R_+\times\R^2}\chi_R\big|\nabla s_\e-(e_i+\nabla\varphi_i)(\tfrac\cdot\e)\nabla_is\big|^2\,=\,0.\]
Recalling $\tilde u_\e=\nabla^\bot s_\e$ and $u=\nabla^\bot s$, the conclusion~\eqref{eq:corr-res-veps} follows.
\end{proof}

In the above, we have used the following convenient trace estimate.

\begin{lem}[Trace estimate]\label{lem:trace}
Let $s\in H^1((I_n+\rho B)\setminus I_n)$ satisfy
\[\left\{\begin{array}{ll}
\triangle s= w,&\text{in $(I_n+\rho B)\setminus I_n$},\\
\int_{\partial I_n}\partial_\nu s=0.&\\
\end{array}\right.\]
Then for all $\chi\in H^1(I_n)$,
\[\Big|\int_{\partial I_n}\chi\partial_\nu s\Big|\,\lesssim\,\|\nabla\chi\|_{\Ld^2(I_n)}\Big(\|\nabla s\|_{\Ld^2((I_n+\rho B)\setminus I_n)}+\| w\|_{\Ld^2((I_n+\rho B)\setminus I_n)}\Big).\]
\end{lem}

\begin{proof}
Consider the auxiliary problem
\begin{equation}\label{eq:aux}
\left\{\begin{array}{ll}
\triangle q=0,&\text{in $I_n$},\\
\partial_\nu q=\partial_\nu s,&\text{on $\partial I_n$},\\
\int_{I_n}q=0.
\end{array}\right.
\end{equation}
We split the proof into two steps.

\medskip
\step1 Well-posedness of~\eqref{eq:aux}: there exists a unique weak solution $q\in H^1(I_n)$ of equation~\eqref{eq:aux} and it satisfies
\begin{equation}\label{eq:apriori-aux}
\|\nabla q\|_{\Ld^2(I_n)}\,\lesssim\,\|\nabla s\|_{\Ld^2((I_n+\rho B)\setminus I_n)}+\| w\|_{\Ld^2((I_n+\rho B)\setminus I_n)}.
\end{equation}
In the weak sense, equation~\eqref{eq:aux} means for all $g\in C^\infty_b(I_n)$,
\[\int_{I_n} \nabla g\cdot\nabla q\,=\,\int_{\partial I_n}g\partial_\nu s.\]
In order to prove the well-posedness and the a priori estimate~\eqref{eq:apriori-aux}, it suffices to show that the linear map
\[L(g)\,:=\,\int_{\partial I_n}g\partial_\nu s\]
satisfies for all $g\in C^\infty_b(I_n)$,
\begin{equation}\label{eq:toprove-L}
|L(g)|\,\lesssim\,\|\nabla g\|_{\Ld^2(I_n)}\Big(\|\nabla s\|_{\Ld^2((I_n+\rho B)\setminus I_n)}+\| w\|_{\Ld^2((I_n+\rho B)\setminus I_n)}\Big).
\end{equation}
For that purpose, we start by considering an extension operator
\[T:C^\infty_b(I_n)\to C^\infty_c(I_n+\rho B),\qquad Tg|_{I_n}=g,\]
which can be chosen to satisfy
\begin{equation}\label{eq:extension-bnd}
\|Tg\|_{H^1(I_n+\rho B)}\,\lesssim\,\|Tg\|_{H^1(I_n)}.
\end{equation}
In view of the assumptions on $s$, the linear map $L$ may then be written as follows: for all $g\in C^\infty_b(I_n)$ and $c\in\R$,
\begin{eqnarray*}
L(g)&=&\int_{\partial I_n}(g-c)\partial_\nu s
\,=\,\int_{\partial I_n}T(g-c)\partial_\nu s\\
&=&-\int_{(I_n+\rho B)\setminus I_n}\nabla T(g-c)\cdot\nabla s
-\int_{(I_n+\rho B)\setminus I_n}T(g-c) w,
\end{eqnarray*}
and thus,
\begin{equation*}
|L(g)|\,\le\,\|T(g-c)\|_{H^1(I_n+\rho B)}\Big(\|\nabla s\|_{\Ld^2((I_n+\rho B)\setminus I_n)}+\| w\|_{\Ld^2((I_n+\rho B)\setminus I_n)}\Big).
\end{equation*}
Using~\eqref{eq:extension-bnd} and the Poincaré inequality with the choice $c:=\fint_{I_n}g$, this yields the claim~\eqref{eq:toprove-L}.

\medskip
\step2 Conclusion.\\
In terms of the solution $q$ of the auxiliary problem~\eqref{eq:aux}, we can write
\[\int_{\partial I_n}\chi\partial_\nu s\,=\,\int_{\partial I_n}\chi\partial_\nu q\,=\,\int_{I_n}\nabla\chi\cdot\nabla q,\]
and the conclusion then follows from~\eqref{eq:apriori-aux}.
\end{proof}

\subsection{Homogenization of the vorticity transport equation}
The strong $\Ld^2$ description~\eqref{eq:corr-res-veps} of the fluid velocity $\tilde u_\e$ is of course not enough to describe actual characteristics of the vorticity transport equation.
Instead, we appeal to almost sure two-scale compactness in the following time-dependent form. In the periodic setting, two-scale compactness originates in the works of Nguetseng and Allaire~\cite{Nguetseng-89,Allaire-92} and first appeared in a time-dependent form in~\cite[Theorem~3.2]{E-92}. In the random setting, the so-called two-scale convergence ``in the mean'' was introduced in~\cite{Bourgeat-Mikelic-Wright-94}, and an almost sure version can be found in~\cite[Section~5]{Zhikov-Pyatnitski-06}. A short proof is included for convenience for the present time-dependent random version.

\begin{lem}[Two-scale compactness]\label{lem:compact-2sc}
Let $\{g_\e\}_\e$ be a sequence in $\Ld^2_\loc(\R_+\times\R^2;\Ld^2(\Omega))$ such that almost surely, for all $T,R>0$,
\[\limsup_{\e\downarrow0}\|g_\e\|_{\Ld^2([0,T]\times B_R)}\,<\,\infty.\]
Then, almost surely, up to a subsequence, we have for any $\Psi\in C^\infty_c(\R_+\times\R^2)$ and any stationary random field $\Theta\in C^\infty_b(\R^2;\Ld^2(\Omega))$,
\[\lim_{\e\downarrow0}\int_{\R_+\times\R^2}g_\e(t,x)\Psi(t,x)\Theta(\tfrac x\e)\,dtdx\,=\,\E\bigg[\Theta(0)\int_{\R_+\times\R^2}G(t,x,0)\Psi(t,x)\,dtdx\bigg],\]
where the two-scale limit $G$ belongs to $\Ld^2_\loc(\R_+\times\R^2;\Ld^2_\loc(\R^2;\Ld^2(\Omega)))$ and $G(t,x,\cdot)$ is a stationary random field for all~$t,x$.
\end{lem}

\begin{proof}
For all $N\ge1$, $\Psi_1,\ldots,\Psi_N\in C^\infty([0,T]\times B_R)$, and stationary random fields $\Theta_1,\ldots,\Theta_N\in C^\infty_b(\R^2;\Ld^2(\Omega))$, the Cauchy--Schwarz inequality yields
\begin{multline*}
\bigg|\int_{[0,T]\times B_R }g_\e(t,x)\Big(\sum_{i=1}^N\Psi_i(t,x)\Theta_i(\tfrac x\e)\Big)dtdx\bigg|\,
\\
\le\,\|g_\e\|_{\Ld^2([0,T]\times B_R)}\bigg(\int_{[0,T]\times B_R}\Big|\sum_{i=1}^N \Psi_i(t,x)\Theta_i(\tfrac x\e)\Big|^2\,dtdx\bigg)^\frac12,
\end{multline*}
and thus, by the ergodic theorem, we get almost surely
\begin{multline}\label{eq:ergth-test-geps}
\limsup_{\e\downarrow0}\bigg|\int_{[0,T]\times B_R}g_\e(t,x)\Big(\sum_{i=1}^N\Psi_i(t,x)\Theta_i(\tfrac x\e)\Big)dtdx\bigg|\\
\,\le\,\Big(\limsup_{\e\downarrow0}\|g_\e\|_{\Ld^2([0,T]\times B_R)}\Big)\bigg(\int_{[0,T]\times B_R}\E\bigg[\Big|\sum_{i=1}^N \Psi_i\Theta_i(0)\Big|^2\bigg]dtdx\bigg)^\frac12.
\end{multline}
Note that the set
\begin{multline*}
\Big\{(t,x,y,\omega)\mapsto\sum_{i=1}^N\Psi_i(t,x)\Theta_i(y,\omega)~:~N\ge1,~\Psi_1,\ldots,\Psi_N\in C^\infty([0,T]\times B_R),\\[-3mm]
\text{and}~\Theta_1,\ldots,\Theta_N\in C^\infty_b(\R^2;\Ld^2(\Omega))~\text{stationary}\Big\}
\end{multline*}
is dense in the Hilbert space
\begin{multline*}
\mathcal H_{T,R}\,:=\,\Big\{H\in\Ld^2([0,T]\times B_R;\Ld^2_\loc(\R^2;\Ld^2(\Omega)))\,:\,\\[-1mm]
\text{$H(t,x,\cdot)$ is a stationary random field for all~$t,x$}\Big\},
\end{multline*}
endowed with the norm
\[\|H\|_{\mathcal H_{T,R}}\,=\,\Big(\int_{[0,T]\times B_R}\E[|H(t,x,0)|^2]\,dtdx\Big)^\frac12.\]
As $\mathcal H_{T,R}$ is separable (using that $\Omega$ is endowed with a countably generated $\sigma$-algebra), the bound~\eqref{eq:ergth-test-geps} ensures by compactness
that, up to a subsequence, we have for all $H\in\mathcal H_{T,R}$,
\[\lim_{\e\downarrow0}\int_{[0,T]\times B_R}g_\e(t,x)H(t,x,\tfrac x\e)\,dtdx\,=\,L_{T,R}(H),\]
for some bounded linear functional $L_{T,R}$ on $\mathcal H_{T,R}$.
By the Riesz representation theorem, $L_{T,R}$ can be represented by an element in $\mathcal H_{T,R}$.
Combining this with a diagonal extraction as $T,R\uparrow\infty$, the conclusion follows.
%\footnote{As the $\sigma$-algebra on $\Omega$ is chosen as the one generated by the random inclusions $\{I_n\}_n$, we recall that for a random variable $g\in\Ld^2(\Omega)$ its stationary extension $g^\sharp$ is defined as follows: for all $x\in\R^2$, as $g$ is $\sigma(\{I_n\}_n)$-measurable, we can define $g^\sharp(x)\in\Ld^2(\Omega)$ to be obtained from $g$ by replacing $\{I_n\}_n$ by $\{x+I_n\}_n$. This obviously defines a stationary random field $g^\sharp\in\Ld^2_\loc(\R^2;\Ld^2(\Omega))$.}
\end{proof}

With the above compactness result at hand, we are in position to establish the following partial homogenization result for the vorticity transport equation.

\begin{prop}\label{prop:transport-imper0}
Almost surely, the extracted limit $(w,u)$ in~\eqref{eq:conv-omu-extr} satisfies the following: there exists $W\in\Ld^2_\loc(\R_+\times\R^2;\Ld^2_\loc(\R^2;\Ld^2(\Omega)))$ such that $W(t,x,\cdot)$ is a stationary random field for almost all $t,x$, and such that
\begin{equation}\label{eq:Weqn}
\left\{\begin{array}{ll}
 w(t,x)=\E[W(t,x,\cdot)],&\text{for almost all $t,x$},\\[1mm]
W(t,x,\cdot)=0,&\text{in $\Ic$, a.s.\@ for almost all $t,x$},\\[1mm]
\Div(W(t,x,\cdot)R_{t,x})=0,&\text{in $\R^2$, a.s.\@ for almost all $t,x$},\\[1mm]
\partial_t w(t,x)+\Div_x\E[W(t,x,\cdot)R_{t,x}]=g,&\text{in $\R^2$},\\
w|_{t=0}= w^\circ,&
\end{array}\right.
\end{equation}
where we have set for abbreviation
\begin{equation}\label{eq:Rtx-short}
R_{t,x}\,:=\,-(e_i+\nabla\varphi_i)^\bot (u(t,x))_i^\bot.
\end{equation}
\end{prop}

\begin{proof}
By Lemma~\ref{lem:compact-2sc}, almost surely, up to extracting a further subsequence, there exists $W\in\Ld^2_\loc(\R_+\times\R^2;\Ld^2_\loc(\R^2;\Ld^2(\Omega)))$ such that $W(t,x,\cdot)$ is a stationary random field for all $t,x$ and such that for all $\Psi\in C^\infty_c(\R_+\times\R^2)$ and all stationary random fields $\Theta\in C^\infty_b(\R^2;\Ld^2(\Omega))$ we have
\begin{equation}\label{eq:2scale-conv}
\int_{\R_+\times\R^2}\tilde w_\e(t,x)\,\Psi(t,x)\Theta(\tfrac x\e)\,dtdx\,\to\,\E\bigg[\Theta(0)\int_{\R_+\times\R^2}W(t,x,0)\,\Psi(t,x)\,dtdx\bigg].
\end{equation}
In particular, the weak limit $w$ is related to the two-scale limit via
\begin{equation*}
 w(t,x)\,=\,\E[W(t,x,\cdot)].
\end{equation*}
In addition, as by definition $\tilde w_\e=0$ in $\e\Ic$, we deduce $W(t,x,\cdot)=0$ in $\Ic$ a.s.\@ for almost all~$t,x$.
Next, recall that the extended vorticity~$\tilde w_\e$ satisfies~\eqref{eq:vort-tildeom-tsp},
\[\partial_t\tilde w_\e+\Div(\tilde u_\e\tilde w_\e)=\tilde g_\e,\qquad\text{in $\R^2$}.\]
The weak formulation of this equation yields the following: for all $\Psi\in C^\infty_c(\R_+\times\R^2)$ and $\Theta\in C^\infty_b(\R^2;\Ld^2(\Omega))$,
\begin{multline}\label{eq:test-ome-Psi}
\int_{\R^2}\tilde w_\e^\circ(x)\,\Psi(0,x)\Theta(\tfrac x\e)\,dx
+\int_{\R_+\times\R^2}\tilde w_\e(t,x)\,(\partial_t\Psi)(t,x)\Theta(\tfrac x\e)\,dtdx\\
+\int_{\R_+\times\R^2}\tilde u_\e(t,x)\tilde w_\e(t,x)\cdot\big(\nabla\Psi(t,x)\Theta(\tfrac x\e)+\tfrac1\e\Psi(t,x)\nabla\Theta(\tfrac x\e)\big)\,dtdx\\
+\int_{\R_+\times\R^2}\tilde g_\e(t,x)\Psi(t,x)\Theta(\tfrac x\e)\,dtdx
\,=\,0.
\end{multline}
On the one hand, multiplying both sides by $\e$ and letting $\e\downarrow0$, we get
\[\int_{\R_+\times\R^2}\tilde u_\e(t,x)\tilde w_\e(t,x)\cdot \Psi(t,x)\nabla\Theta(\tfrac x\e)\,\to\,0,\]
and thus, using~\eqref{eq:corr-res-veps} and~\eqref{eq:2scale-conv},
\[\int_{\R_+\times\R^2}\Psi(t,x)\,\E\big[\nabla\Theta(0)\cdot (e_i+\nabla\varphi_i)^\bot W(t,x,0)\big]\,u_i^\bot(t,x)\,dtdx\,=\,0.\]
By the arbitrariness of $\Psi,\Theta$, this means that for almost all $t,x$ the stationary random field $W(t,x,\cdot)\in\Ld^2_\loc(\R^2;\Ld^2(\Omega))$ satisfies
\begin{equation*}
\Div( W(t,x,\cdot)R_{t,x})\,=\,0,
\end{equation*}
with the short-hand notation $R_{t,x}$, cf.~\eqref{eq:Rtx-short}. 
On the other hand, choosing $\Theta=1$ in~\eqref{eq:test-ome-Psi}, and using~\eqref{eq:2scale-conv} to pass to the limit, we further get the desired macroscopic transport equation in~\eqref{eq:Weqn}.
\end{proof}

The third equation in~\eqref{eq:Weqn} means that for almost all $t,x$ the stationary random field~$W(t,x,\cdot)$ is an invariant measure for the dynamics generated by the vector field $R_{t,x}$.
In general, even in the periodic setting, such invariant measures might not be unique, so that the limit system~\eqref{eq:Weqn} would not characterize the limit uniquely.
% since invariant measures might not be unique for the dynamics generated by $R_{t,x}$ --- even in the periodic setting.
We refer to~\cite{Jabin-Tzavaras-09} for possible pathological behaviors in the homogenization of multiscale transport equations.
More precisely, the issue is not exactly the non-uniqueness of the invariant measure, but rather the fact that the Herman rotation set
\[C_{t,x}\,:=\,\Big\{\E[\mu R_{t,x}]\,:\,\mu\in\Ld^2(\Omega)~\text{is invariant under the flow generated by $R_{t,x}$}\Big\}\]
might not reduce to a singleton.
As shown in~\cite[Theorem~1.2]{Franks-Misiurewicz-90}, in the 2D periodic setting, this rotation set is either a singleton, or a closed line segment passing through $0$ with rational slope, or a closed line segment rooted at $0$ with irrational slope.
As noted in~\cite{Briane-Herve-21,Briane-Herve-22}, this allows to project the drift in this direction in the homogenized transport equation, but it still leaves a 1D indeterminacy in general.
To get beyond this, a more detailed analysis of the dynamics is necessary.

When the vector field is incompressible, as it is the case for $R_{t,x}$ here, the behavior of the generated dynamics is in fact well understood in the 2D periodic setting. Indeed, as shown by Arnol$'$d~\cite{Arnold-91}, provided that the two frequencies of the vector field are incommensurable, there necessarily exist finitely many domains $\{U_k\}_{1\le k\le r}$ in the periodicity cell such that the trajectories in each $U_k$ are either periodic or tend to a fixed point, and such that outside of those domains the trajectories form one single ergodic class. In particular, if the incompressible vector field has no fixed point, then we recover the standard result that the dynamics admits a unique invariant measure whenever its two frequencies are incommensurable; see e.g.~\cite[Section~2]{E-92}. In the presence of fixed points, on the contrary, part of the mass can get trapped into bounded periodic trajectories at the microscale, thus leading in the macroscopic limit to a localization phenomenon. The remaining question is thus twofold:
\begin{enumerate}[(a)]
\item Does the flow generated by the vector field $R_{t,x}$ admit a fixed point (or, equivalently, does the vector field $R_{t,x}$ vanish at a point)?
This is the case if there exists a direction $e\in\Sp^1$ such that the corrected harmonic coordinate $x\mapsto e\cdot x+\varphi_e(x)$ admits a critical point. If it does, then a localization phenomenon would appear for the macroscopic vorticity.
\smallskip\item
In the stationary random setting, is the absence of fixed points of $R_{t,x}$ enough to guarantee the uniqueness of the invariant measure?
\end{enumerate}
Question~(a) is addressed in Section~\ref{sec:harmcoord}, where the localization phenomenon is excluded in full generality: by a post-processing of the work of  Alessandrini and Nesi~\cite{Alessandrini-Nesi-01,Alessandrini-Nesi-18}, we show that in 2D corrected harmonic coordinates $x\mapsto x+\varphi(x)$ define a diffeomorphism of~$\R^2\setminus\Ic$ onto its image, cf.~Theorem~\ref{th:diffeo-rigid}, which ensures that the vector field $R_{t,x}$ cannot have critical points. This is proven to hold even in the general stationary random setting.
In contrast, question~(b) is more subtle and we do not yet have a definite answer, cf.~Question~\ref{qu:unique} above, except in the particular setting of Lemma~\ref{lem:unique}.

\subsection{Conclusion}
Combining the above results, we are now in position to conclude the proof of Theorems~\ref{th:imper} and~\ref{th:imper-rand}. We start with the latter.

\begin{proof}[Proof of Theorem~\ref{th:imper-rand}]
Let $(w,u)$ be the extracted limit in~\eqref{eq:conv-omu-extr}.
By Proposition~\ref{prop:transport-imper0},
there exists $W\in\Ld^2_\loc(\R_+\times\R^2;\Ld^2_\loc(\R^2;\Ld^2(\Omega)))$ such that $W(t,x,\cdot)$ is a stationary random field for all $t,x$ and such that~\eqref{eq:Weqn} is satisfied. In particular, for almost all $t,x$, almost surely, $W(t,x,\cdot)$ is an invariant measure for the dynamics generated by the vector field~$R_{t,x}$ defined in~\eqref{eq:Rtx-short}. Assuming that the answer to Question~\ref{qu:unique} is positive, this entails that $W(t,x,\cdot)$ is almost surely constant in $\R^2\setminus\Ic$ for almost all $t,x$. As~\eqref{eq:Weqn} yields $W(t,x,\cdot)=0$ in $\Ic$ and $\E[W(t,x,\cdot)]=w(t,x)$, we deduce
\[W(t,x,\cdot)\,=\,\tfrac{1}{1-\lambda}\,w(t,x)\mathds1_{\R^2\setminus\Ic},\]
in terms of the inclusions' volume fraction $\lambda:=\E[\mathds1_\Ic]$.
Hence, by definition~\eqref{eq:Rtx-short} of $R_{t,x}$, we find
\[\E[W(t,x,\cdot)R_{t,x}]\,=\,\tfrac{1}{1-\lambda}w(t,x)u(t,x),\]
and equation~\eqref{eq:Weqn} then reduces to the following transport equation for $w$,
\begin{equation*}
\left\{\begin{array}{ll}
\partial_t w+\frac1{1-\lambda}\Div(wu)=g,&\text{in $\R_+\times\R^2$},\\[1mm]
w|_{t=0}=w^\circ.&
\end{array}\right.
\end{equation*}
Combining this with Proposition~\ref{prop:omu-lim} and noting that the uniqueness of the solution of the limit problem allows to get rid of the extraction of a subsequence, this concludes the proof.
\end{proof}

We now turn to the proof of Theorem~\ref{th:imper} in the periodic setting. Recall that the periodic setting can be viewed as a particular case of our stationary random setting, by setting $\Omega=Q$ the periodicity cell and by replacing expectation by periodic averaging.

\begin{proof}[Proof of Theorem~\ref{th:imper}]
Let $(w,u)$ be the extracted limit in~\eqref{eq:lim-omu-res}, cf.~\eqref{eq:conv-omu-extr}.
Repeating the proof of Proposition~\ref{prop:omu-lim}, we obtain that $u$ satisfies
\begin{equation}\label{eq:lim-divcurl-per}
\left\{\begin{array}{rcll}
\Div(u)&=&0,&\text{in $\R^2$},\\[1mm]
\curl(\bm{\bar m} u)&=&w,&\text{in $\R^2$},
\end{array}\right.\end{equation}
where $\bm{\bar m}$ is now the symmetric matrix defined in~\eqref{eq:def-barm}--\eqref{eq:def-bara}. Recall that this ensures $u\in\Ld^\infty_\loc(\R_+;W^{1,p}\cap W^{r,\infty}(\R^2)^2)$ for all $2\le p<\infty$ and $r<1$.
In addition, repeating the proof of Proposition~\ref{prop:transport-imper0}, we can show that there exists $W\in \Ld^2_\loc(\R_+\times\R^2;\Ld^2_\per(Q))$ such that the limit vorticity $ w$ satisfies
\begin{equation}\label{eq:Weqn-per}
\left\{\begin{array}{ll}
 w(t,x)=\int_QW(t,x,\cdot),&\text{for almost all $t,x$},\\[1mm]
W(t,x,\cdot)=0,&\text{in $\Ic$, for almost all $t,x$},\\[1mm]
\Div(W(t,x,\cdot)R_{t,x})=0,&\text{in $Q\setminus\Ic$, for almost all $t,x$},\\[1mm]
\partial_t w(t,x)+\Div_x\big(\int_QW(t,x,\cdot)R_{t,x}\big)=g,&\text{in $\R_+\times\R^2$},\\[1mm]
w|_{t=0}= w^\circ,&
\end{array}\right.
\end{equation}
where we have set for abbreviation
\[R_{t,x}\,:=\,-(e_i+\nabla\varphi_i)^\bot (u(t,x))_i^\bot~\in~\Ld^2_\per(Q)^2,\]
where $\varphi_i\in H^1_\per(Q)$ now stands for the periodic corrector~\eqref{eq:varphi-imperm}. In order to conclude, as in the proof of Theorem~\ref{th:imper-rand} above, it suffices to show that $W(t,x,\cdot)$ is constant in $Q\setminus \Ic$.

For that purpose, we start by checking that the frequencies of $R_{t,x}$ are almost never in resonance. More precisely, we show that the set
\begin{equation}\label{eq:badset-N}
N\,:=\,\big\{(t,x)\in\R_+\times\R^2:\exists (e_1,e_2)\in\R^2, \,e_1/e_2\in\Q,\,u(t,x)=e\big\}
\end{equation}
has zero Lebesgue measure.
Recall that $u$ belongs to $\Ld^\infty_\loc(\R_+;W^{1,p}\cap W^{r,\infty}(\R^2))$ for all $2\le p<\infty$ and $r<1$. The evolution equation for $w$ in~\eqref{eq:Weqn-per} then yields in particular $\partial_t w\in \Ld^\infty_\loc(\R_+;\dot H^{-1}(\R^2))$. Taking the time derivative in the limit div-curl problem~\eqref{eq:lim-divcurl-per}, this easily implies $\partial_tu\in \Ld^\infty_\loc(\R_+;\Ld^2(\R^2))$. Therefore, by interpolation, we can conclude that $u$ belongs to $C_\loc(\R_+;C^r_\loc(\R^2))$ for all $r<1$. Given this regularity, if the set~\eqref{eq:badset-N} did not have zero Lebesgue measure, there would exist a nonempty open set $U\subset\R_+\times\R^2$, a vector $(e_1,e_2)\in\R^2$ with $e_1/e_2\in\Q$, and a continuous scalar function $h:U\to\R$ such that $u(t,x)=eh(t,x)$ for all $(t,x)\in U$. Now, as $u=\nabla^\bot s$ for some stream function $s$, this relation would imply that the function $h$ is of the form $h(t,x)=h_0(t,e^\bot\cdot x)$ on~$U$, which is precisely excluded by the non-degeneracy assumption~\eqref{eq:non-deg-re}.

Let us now turn back to the analysis of~\eqref{eq:Weqn-per}.
For $(t,x)\notin N$, the vector field $R_{t,x}$ takes the form $(e+\nabla\varphi_e)^\bot$ for some $e=(e_1,e_2)$ with $e_1/e_2\notin\Q$.
As the vector field $(e+\nabla\varphi_e)^\bot$ is periodic, divergence-free, smooth in $Q\setminus\overline \Ic$, as it is tangential to the boundary $\partial\Ic$, as it has incommensurate frequencies in the sense that
\[\int_Q(e+\nabla\varphi_e)^\bot=e^\bot=(-e_2,e_1),\qquad e_1/e_2\notin\Q,\]
and as it has no singular points in the sense that $|e+\nabla\varphi_e|>0$ on $Q\setminus\overline \Ic$ by Theorem~\ref{th:diffeo-rigid},
it is well known (see e.g.~\cite[Section~2]{E-92}) that the associated flow is ergodic outside $\overline\Ic$, with a unique invariant measure that is the Lebesgue measure restricted to $Q\setminus\overline \Ic$. In other words, any periodic measure solution~$\mu$ of the equation
\[\Div(\mu(e+\nabla\varphi_e)^\bot)=0,\qquad\text{in $Q\setminus\Ic$},\]
is proportional to the Lebesgue measure in $Q\setminus\Ic$.
From~\eqref{eq:Weqn-per}, we may thus deduce that $W(t,x,\cdot)$ is constant in $\R^2\setminus\Ic$ for all $(t,x)\notin N$.
As $N$ is a null set, this concludes the proof.
\end{proof}

\section{Homogenization of the lake equations}
This section is devoted to the proof of Theorem~\ref{th:lake}  in the periodic setting. As it follows along the same lines as the proof of Theorem~\ref{th:imper} above, we only give a short sketch.
Let $b:\R^2\to\R_+$ be a periodic function that satisfies~\eqref{eq:unifell} and that is locally Hölder-continuous.

With the assumptions on the data $u_\e^\circ,f_\e$, cf.~\eqref{eq:conv-data-lake},
the lake equations~\eqref{eq:lake} admit a unique global weak solution $u_\e\in\Ld^\infty_\loc(\R_+;\Ld^2(\R^2)^2)$ with vorticity $w_\e:=\curl(u_\e)\in\Ld^\infty_\loc(\R_+;\Ld^1\cap\Ld^\infty(\R^2))$, cf.~\cite{Levermore-Oliver-Titi-1,Levermore-Oliver-Titi-2}.
In addition, in terms of $w_\e^\circ:=\curl(u_\e^\circ)$ and $g_\e:=\curl(f_\e)$, the lake equations may alternatively be written in vorticity formulation,
\begin{equation}\label{eq:imper-vort}
\left\{\begin{array}{ll}
\partial_tw_\e+\Div(u_\e w_\e)=g_\e,&\text{in $\R^2$},\\
\Div(b(\tfrac\cdot\e)u_\e)=0,&\text{in $\R^2$},\\
\curl(u_\e)=w_\e,&\text{in $\R^2$},\\
w_\e|_{t=0}=w_\e^\circ,&
\end{array}\right.
\end{equation}
and we have the following a priori estimates, for all $t\ge0$,
\begin{eqnarray*}
\|w_\e(t)\|_{\Ld^1\cap\Ld^\infty(\R^2)}&\lesssim&\|w_\e^\circ\|_{\Ld^1\cap\Ld^\infty(\R^2)}+\|g_\e\|_{\Ld^1(0,t;\Ld^1\cap\Ld^\infty(\R^2))},\\
\|u_\e(t)\|_{\Ld^2(\R^2)}&\lesssim&\|u_\e^\circ\|_{\Ld^2(\R^2)}+\|f_\e\|_{\Ld^1(0,t;\Ld^2(\R^2))},
\end{eqnarray*}
where the multiplicative constants only depend on the constant $C_0$ in~\eqref{eq:unifell}, and
where the right-hand sides are uniformly bounded as $\e\downarrow0$ by assumption~\eqref{eq:conv-data-lake}.
By weak compactness, we deduce, up to a subsequence,
\begin{equation}\label{eq:conv-omu-extr-lake}
\begin{array}{rcll}
w_\e&\cvfs&w,\qquad&\text{in $\Ld^\infty_\loc(\R_+;\Ld^1\cap\Ld^\infty\cap\dot H^{-1}(\R^2))$},\\
u_\e&\cvfs& u,&\text{in $\Ld^\infty_\loc(\R_+;\Ld^2(\R^2))$},\\
b(\tfrac\cdot\e)u_\e&\cvfs& v,&\text{in $\Ld^\infty_\loc(\R_+;\Ld^2(\R^2))$},
\end{array}\end{equation}
for some limit $(w,u,v)$, which remains to be characterized.
As in Section~\ref{sec:homog-imper}, we start with the homogenization of the div-curl problem for $u_\e$ in~\eqref{eq:imper-vort}. In the present case, the div-curl problem is equivalent to a scalar uniformly-elliptic problem and we can appeal to standard homogenization theory.

\begin{prop}\label{prop:lake-1}
The extracted limit $(w,u,v)$ in~\eqref{eq:conv-omu-extr-lake} satisfies $v=\bb u$ and the following properties.
\begin{enumerate}[(i)]
\item \emph{Homogenized div-curl problem:}
\begin{equation*}
\left\{\begin{array}{rcll}
\Div(\bb u)&=&0,&\text{in $\R^2$},\\[1mm]
\curl(u)&=&w,&\text{in $\R^2$},
\end{array}\right.\end{equation*}
where $\bb\in\R^{2\times 2}$ is the symmetric matrix defined in~\eqref{eq:def-barb-lake}--\eqref{eq:def-cor-lake}. In particular, this implies additional regularity: the limit $u$ necessarily belongs to $\Ld^\infty_\loc(\R_+;W^{1,p}\cap W^{r,\infty}(\R^2)^2)$ for all $2\le p<\infty$ and $r<1$.
\smallskip\item \emph{Corrector result:} along the extracted subsequence, we have the strong convergence
\begin{equation*}
u_\e+(b^{-1}(e_i+\nabla\psi_i)^\bot)(\tfrac\cdot\e)\,(\bb u)^\bot_i~\to~0,\qquad\text{in $\Ld^\infty_\loc(\R_+;\Ld^2_{\loc}(\R^2))$,}
\end{equation*}
where $\psi_i\in H^1_\per(\R^2)$ stands for the periodic corrector~\eqref{eq:def-cor-lake} in the direction $e_i$.
\end{enumerate}
\end{prop}

\begin{proof}
By the incompressibility condition for the fluid velocity $u_\e$ in~\eqref{eq:imper-vort}, there exists a stream function $s_\e\in\Ld^\infty_\loc(\R_+;\dot H^1(\R^2))$ such that $u_\e=b(\tfrac\cdot\e)^{-1}\nabla^\bot s_\e$, and the div-curl problem for~$u_\e$ in~\eqref{eq:imper-vort} then takes form of the following uniformly elliptic problem for $s_\e$,
\[\Div( b(\tfrac\cdot\e)^{-1}\nabla s_\e) = w_\e,\qquad\text{in $\R^2$}.\]
The conclusion then follows from standard periodic homogenization, e.g.~\cite{MuratTartar,JKO94}.
\end{proof}

With the above strong $\Ld^2$ description of the fluid velocity $u_\e$ at hand, we can now turn to the two-scale analysis of the vorticity transport equation. Arguing similarly as for Proposition~\ref{prop:transport-imper0}, now using standard two-scale compactness as in~\cite[Theorem~3.2]{E-92}, we obtain the following. We skip the proof for conciseness.

\begin{prop}\label{prop:lake-2}
The extracted limit $(w,u)$ in~\eqref{eq:conv-omu-extr-lake} satisfies the following: there exists $W\in\Ld^2_\loc(\R_+\times\R^2;\Ld^2_\per(Q))$ such that
\begin{equation}\label{eq:W-syst-lake}
\left\{\begin{array}{ll}
w(t,x)=\int_QW(t,x,\cdot),&\text{for almost all $t,x$},\\[1mm]
\Div(W(t,x,\cdot)R_{t,x})=0,&\text{in $Q$, for almost all $t,x$},\\[1mm]
\partial_tw(t,x)+\Div_x\big(\int_QW(t,x,\cdot)R_{t,x}\big)=g,&\text{in $\R^2$},\\
w|_{t=0}=w^\circ,&
\end{array}\right.
\end{equation}
where we have set for abbreviation
\begin{equation}\label{eq:def-Rtx-per}
R_{t,x}\,:=\,-b^{-1}(e_i+\nabla\varphi_i)^\bot\,(\bb u(t,x))^\bot_i~\in~\Ld^2_\per(\R^2)^2.
\end{equation}
\end{prop}

In terms of $W'(t,x,y):=b(y)^{-1}W(t,x,y)$ and $R_{t,x}'(y):=b(y)R_{t,x}(y)$, the invariant measure equation for $W$ in~\eqref{eq:W-syst-lake} reads
\[\Div(W'(t,x,\cdot)R_{t,x}')=0,\qquad \text{in $Q$, for almost all $t,x$}.\]
Note that $R_{t,x}'$ is incompressible and has frequencies $\fint_QR_{t,x}'=\bb u(t,x)=v(t,x)$. Arguing similarly as in the proof of Theorem~\ref{th:imper}, under the non-degeneracy condition~\eqref{eq:non-deg-re-lake}, we then deduce that $W'(t,x,\cdot)$ is constant on $Q$, which means that $W(t,x,\cdot)$ is proportional to $b^{-1}$ on $Q$. Hence, as~\eqref{eq:W-syst-lake} requires $w(t,x)=\int_QW(t,x,\cdot)$, we get
\[W(t,x,y)\,=\,\Big(\frac{b(y)}{\int_Qb}\Big)w(t,x).\]
Inserting this in the equation for $w$ in Proposition~\ref{prop:lake-2}, and combining with Proposition~\ref{prop:lake-1}, this concludes the proof of Theorem~\ref{th:lake}.\qed

\section{Invertibility of corrected harmonic coordinates}\label{sec:harmcoord}

The main result of this section is the invertibility of corrected harmonic coordinates, which was used in the previous sections as the key ingredient to avoid vorticity localization in the homogenization limit.
In the case of uniformly elliptic and bounded coefficient fields, the invertibility of corrected coordinates was established by Alessandrini and Nesi~\cite{Alessandrini-Nesi-01,Alessandrini-Nesi-18} in the periodic setting. Using the (purely qualitative) large-scale Lipschitz estimates of~\cite{MR4103433}, we show that this can be extended to the general stationary random setting. 

\begin{theor}\label{th:diffeo}
Let $\Aa$ be a stationary ergodic random field on $\R^2$ taking values in a set of uniformly elliptic and bounded symmetric matrices, and assume that $\Aa$ is almost surely locally H\"older-continuous. For $1\le i\le2$, denote by $\psi_i$ the unique almost sure weak solution of
\[\Div(\Aa(\nabla\psi_i+e_i))=0,\qquad\text{in $\R^2$},\]
such that $\nabla\psi_i$ is stationary, $\E[\nabla\psi_i]=0$, and with anchoring~$\psi_i(0)=0$.
Set $\psi:=(\psi_1,\psi_2)$.
Then, almost surely, the map $\R^2\to \R^2:x \mapsto x + \psi(x)$ is a  diffeomorphism on $\R^2$.
\end{theor}

\begin{proof}
We split the proof into two steps. In the first step, we show that one can approximate the map $\Psi(x):=x+\psi(x)$ on bounded domains using qualitative large-scale Lipschitz regularity. In the second step, we then use this approximation result to extend the results by Alessandrini and Nesi~\cite{Alessandrini-Nesi-01,Alessandrini-Nesi-18} from bounded domains (or periodic coefficients) to the whole space (or random coefficients).

\medskip
\step1 Approximation of $\psi_i$ on bounded domains.\\
Denoting by $B_R$ the ball of size $R>0$ centered at the origin in $\R^2$, we approximate the corrector~$\psi_i$ by the Dirichlet corrector $\psi_{R,i}$ defined as the unique almost sure weak solution in $H^1_0(B_R)$ of
\[\Div(\Aa(\nabla\psi_{R,i}+e_i))=0.\]
We shall show that $\nabla \psi_{R,i}$ converges strongly to $\nabla\psi_i$ in $\Ld^2_\loc(\R^2)$ almost surely. We split the proof into two further substeps.

\medskip
\substep{1.1}
Proof that almost surely
\begin{equation}\label{eq:corr-approx-R}
\lim_{R \uparrow\infty} \fint_{B_{R/2}} | \nabla (\psi_i-\psi_{R,i})|^2=0.
\end{equation}
For $R,K\ge 2$, consider a cut-off function $\eta_{R,K}$ such that
\[\eta_{R,K} = 1~~\text{on $B_{R(1-1/K)}$},\qquad \eta_{R,K}=0~~\text{on $\partial B_R$},\qquad\text{and}\qquad |\nabla \eta_{R,K}|\lesssim KR^{-1}.\]
A direct calculation shows that the function $\psi_{R,K,i}:=\eta_{R,K} \psi_i-\psi_{R,i}$ satisfies in $B_R$,
\[
-\Div(\Aa\nabla \psi_{R,K,i})\,=\,-\Div (\Aa \psi_i \nabla \eta_{R,K})-\Div( (1-\eta_{R,K})\Aa e_i)+\Div((1-\eta_{R,K})\Aa(\nabla \psi_i+e_i)).
\]
Testing with $\psi_{R,K,i} \in H^1_0(B_R)$ yields
\begin{eqnarray*}
\int_{B_R} |\nabla \psi_{R,K,i}|^2\,\lesssim\, K^2 \int_{B_R} R^{-2}|\psi_i|^2  + \int_{B_R \setminus B_{R(1-1/K)}}   (|\nabla \psi_i|^2+1).
\end{eqnarray*}
By the sublinearity of $\psi_i$ at infinity, we have almost surely
\[
\limsup_{R\uparrow\infty} \fint_{B_R} R^{-2}|\psi_i|^2=0.
\]
By the ergodic theorem, we also have almost surely
\[
\lim_{R\uparrow\infty}\fint_{B_R} (|\nabla \psi_i|^2+1) = \expec{|\nabla \psi_i|^2+1},
\]
and thus
\begin{eqnarray*}
\lim_{R\uparrow\infty} |B_R|^{-1} \int_{B_R \setminus B_{R(1-1/K)}}   (|\nabla \psi_i|^2+1)
&=&  (1-(1-\tfrac1K)^2) \expec{|\nabla \psi_i|^2+1}\\
&\lesssim& \tfrac1K\expec{|\nabla \psi_i|^2+1}.
\end{eqnarray*}
This entails the almost sure convergence
\[
\limsup_{K\uparrow\infty} \limsup_{R\uparrow\infty} \fint_{B_R} | \nabla \psi_{R,K,i}|^2\,=\,0,
\]
from which the claim~\eqref{eq:corr-approx-R} follows.

\medskip
\substep{1.2} Proof that $\nabla \psi_{R,i}\to\nabla \psi_i$ in $\Ld^2_\loc(\R^2)$ almost surely.\\
By \cite[Theorem~1]{MR4103433}, there exist an almost surely finite random radius $r_*\ge 1$ and a deterministic constant $C<\infty$ with the following property: almost surely, for all $r\ge \rho \ge r_*$ and all functions $w \in H^1(B_r)$ that satisfy $\Div(\Aa\nabla w)=0$, we have 
\[
\fint_{B_\rho} |\nabla w|^2 \le C \fint_{B_r} |\nabla w|^2.
\]
Applying this to $w=\psi_{R,i}-\psi_i$ on $B_R$, and recalling~\eqref{eq:corr-approx-R}, we deduce for all $\rho \ge r_*$,
\[
\lim_{R \uparrow\infty} \int_{B_\rho} |\nabla (\psi_{R,i}-\psi_i)|^2=0,
\]
and the claim follows.

\medskip
\step2 Conclusion.\\
We aim to show that $\Psi(x)=x+\psi(x)$ is a diffeomorphism on $\R^2$.
First note that it is enough to show that it is a local homeomorphism. Indeed, by~\cite[Theorem~1.1]{Alessandrini-Nesi-18}, as $\Psi$ is $\Aa$-harmonic and as $\Aa$ is almost surely locally H\"older-continuous, the local univalence of~$\Psi$  implies that it is a local diffeomorphism. Since a local diffeomorphism $\R^2 \to \R^2$ is necessarily global, this  concludes the argument.

It remains to argue that $\Psi$ is locally univalent.
To this aim, we need to recall a few results from complex analysis (for which we refer to \cite{Alessandrini-Nesi-01} and references therein).
Let $B$ be a ball centered at $0$ in $\R^2$, let $B_+$ be an open set containing $\overline B$, and let $u \in H^1(B_+)$ be $\Aa$-harmonic in~$B_+$. Then, there exists a unique quasi-conformal map $\chi$ and a unique holomorphic function $H$ such that $u= h \circ \chi$ on $B$, where~$h$ denotes the real part of~$H$. We say that a point $z \in B$ is a geometric critical point of $u$ if $\nabla h (\chi(z))=0$. Geometric critical points are   isolated. 
First, we recall the following result relating the non-existence of critical points and the local univalence of $\Psi$, cf.~\cite[Theorem~3]{Alessandrini-Nesi-01}:
\begin{enumerate}[\qquad]
\item If for all $e \in \R^2\setminus\{0\}$ the $\Aa$-harmonic function $x \mapsto e\cdot \Psi(x)$ has no geometric critical point in $B$, then $\Psi$ is locally univalent on $B$.
\end{enumerate}
Second, we recall a result that allows to estimate the number of geometric critical points by approximation (a direct consequence of~\cite[Lemma~1]{Alessandrini-Nesi-01}):
\begin{enumerate}[\qquad]
\item Let $u_n$ be a sequence of $\Aa$-harmonic functions in $B$ and $u$ be $\Aa$-harmonic in $B$ such that $\nabla u_n \to \nabla u$ in $\Ld^2(B)$. If for all $n$ the function $u_n$ has no geometric critical point in $B$, then the limit~$u$ has no geometric critical point in $B$ either.
\end{enumerate}
The final argument goes as follows.
For all $R>0$, \cite[Theorem~4]{Alessandrini-Nesi-01} entails that $\Psi_R$ is a homeomorphism from $B_R$ to itself. By \cite[Theorem~3]{Alessandrini-Nesi-01}, this implies that $\Psi_R$ has no geometric critical point in any ball $B\subset B_R$. By Step~1, $\nabla \Psi_R$ converges to $\nabla \Psi$ in $\Ld^2(B)$. Hence, by the above, $\Psi$ has no geometric critical point in $B$, and it is therefore locally univalent on~$B$. Since the radius of $B$ is arbitrary, $\Psi$ is locally univalent on $\R^2$, which concludes the proof.
\end{proof}

In order to deal with corrected coordinates associated to the Euler equations in a porous medium, we need to weaken the uniform ellipticity assumptions on the coefficient field $\Aa$ and to allow it to be infinite on a stationary set of inclusions:  the corrector equations  take form of an elliptic problem with rigid inclusions in that case, cf.~\eqref{eq:varphi-imperm}. This extension is the object of the following result.

\begin{theor}\label{th:diffeo-rigid}
Let $\{I_n\}_n$  be a random collection of subsets of $\R^2$ such that the random set $\mathcal I=\bigcup_nI_n$ is stationary and ergodic, and such that for some $\rho>0$ it almost surely satisfies the same regularity and hardcore conditions as in Theorem~\ref{th:imper}.
For $1\le i\le2$, define the corrector $\varphi_i\in H^1_\loc(\R^2;\Ld^2(\Omega))$ as the unique almost sure weak solution of~\eqref{eq:varphi-imperm}, that is,
\begin{equation}\label{eq:varphi-imperm-re}
\left\{\begin{array}{ll}
-\triangle\varphi_i=0,&\text{in $\R^2\setminus\Ic$},\\
\nabla\varphi_i+e_i=0,&\text{in $\Ic$},\\
\int_{\partial I_n}\partial_\nu\varphi_i=0,&\text{for all $n$},
\end{array}\right.
\end{equation}
such that $\nabla\varphi_i$ is stationary, $\E[\nabla\varphi_i]=0$, and with anchoring~$\varphi_i(0)=0$; see e.g.~\cite[Section~8.6]{JKO94}.
Set $\varphi:=(\varphi_1,\varphi_2)$.
Then, almost surely, the map $\R^2\setminus \mathcal I\to \R^2:x \mapsto x + \varphi(x)$ is a   diffeomorphism onto its image.
\end{theor}

\begin{proof}
Similarly as for Theorem~\ref{th:diffeo}, the proof is a post-processing of the work of Alessandrini and Nesi~\cite{Alessandrini-Nesi-01,Alessandrini-Nesi-18}. More precisely, it involves two approximation arguments, combined with an additional short argument for  injectivity.
We split the proof into three steps.

\medskip
\step1 Approximation of $\varphi_i$ on bounded domains with relaxed stiffness.\\
For $R\ge 1$, setting $N_R:=\{n : I_n \subset B_R,\,\dist(I_n,\partial B_R)\ge 1\}$, we approximate the corrector~$\varphi_{i}$ by the Dirichlet corrector $\varphi_{R,i}$ defined as the unique almost sure weak solution in~$H^1_0(B_R)$ of
\begin{equation}\label{eq:varphi-imperm-re-R}
\left\{\begin{array}{ll}
-\triangle\varphi_{R,i}=0,&\text{in $B_R\setminus\cup_{n\in N_R}I_n$},\\
\nabla\varphi_{R,i}+e_i=0,&\text{in $\cup_{n\in N_R}I_n$},\\
\int_{\partial I_n}\partial_\nu\varphi_{R,i}=0,&\text{for all $n\in N_R$}.
\end{array}\right.
\end{equation}
We further introduce another approximation, which amounts to replacing stiff inclusions by inclusions with diverging conductivity.
For $K\ge 1$, due to the uniform regularity assumption for the inclusions,
we can choose for all $n$ a smooth cut-off function $\chi_{K,n}:\R^2 \to [0,1]$ for $I_n$ in~$\R^2$ such that
\begin{gather*}
\text{$\chi_{K,n}(x)=1$ for all $x\in I_n$ with $\dist(x,\partial I_n)\ge\tfrac1K$,}\\
\text{$\chi_{K,n}|_{\R^2 \setminus I_n}=0$, and $|\nabla \chi_{K,n}|\lesssim K$.}
\end{gather*}
Note that those cut-off functions can be chosen such that $\chi_{K,n}$ is increasing wrt $K$.
We then consider the coefficient field
\[\Aa_{R,K}:=\Id + K\Id \sum_{n \in N_R} \chi_{K,n}.\]
and for $1\le i\le2$ we define the corrector $\varphi_{R,K,i}$ as the unique almost sure weak solution in $H^1_0(B_R)$ of
\[\Div(\Aa_{R,K}(\nabla\varphi_{R,K,i}+e_i))=0.\]
We shall show that $\nabla \varphi_{R,K,i}$ converges strongly to $\nabla\varphi_{i}$ in $\Ld^2_\loc(\R^2)$ as $R,K\uparrow\infty$ almost surely. We split the proof into three further substeps.

\medskip
\substep{1.1} Proof that $\varphi_{R,K,i}\to \varphi_{R,i}$ in $H^1_0(B_R)$ and in $C^1(\overline{B_R}\setminus\Ic_R)$ as $K\uparrow\infty$.\\
We appeal to a $\Gamma$-convergence argument. Consider the integral functional
\[J_{R,K}(u):=\int_{B_R} \nabla u \cdot \Aa_{R,K} \nabla u,\qquad u\in H^1(B_R).\]
Since $J_{R,K}$ is increasing wrt $K$, it $\Gamma$-converges as $K\uparrow\infty$ (wrt to the weak topology of~$H^1(B_R)$)
to the integral functional defined by
\[J_{R}(u)\,:=\, \lim_{K\uparrow\infty} J_{R,K}(u),\qquad u\in H^1(B_R).\]
By definition of $\Aa_{R,K}$, we actually find, for $u\in H^1(B_R)$,
\[J_R(u)\,=\,\left\{\begin{array}{lll}
\int_{B_R\setminus\Ic_R}|\nabla u|^2&:&\nabla u|_{\Ic_R}=0,\\
\infty&:&\nabla u|_{\Ic_R}\ne0.
\end{array}\right.\]
In particular, the $\Gamma$-convergence result entails that the minimal energy
\[E_{R,K,i}\,:=\,\inf\{ J_{R,K} (u) \,:\, u(x)= x_i+v(x), \, v\in H^1_0(B_R)\}\]
converges to
\[E_{R,i}\,:=\,\inf\{ J_{R} (u) \,:\, u(x)= x_i+v(x), \,v\in H^1_0(B_R)\},\]
and that the corresponding minimizers converge wrt the weak $H^1$ topology.
By definition, the (unique) minimizers for $E_{R,K,i}$ and $E_{R,i}$ coincide with the correctors~$\varphi_{R,K,i}$ and~$\varphi_{R,i}$, respectively, so that we deduce $\varphi_{R,K,i}\cvf\varphi_{R,i}$ in $H^1_0(B_R)$.
In order to prove strong convergence, we start with the following computation, using $(\nabla\varphi_{R,i}+e_i)|_{\Ic_R}=0$,
\begin{eqnarray*}
\lefteqn{\int_{B_R} |\nabla (\varphi_{R,K,i}-\varphi_{R,i})|^2 ~\le ~\int_{B_R} \nabla (\varphi_{R,K,i}-\varphi_{R,i}) \cdot \Aa_{R,K} \nabla (\varphi_{R,K,i}-\varphi_{R,i})}
\\
&\qquad\qquad=& \int_{B_R} (\nabla\varphi_{R,K,i}+e_i)\cdot \Aa_{R,K}(\nabla\varphi_{R,K,i}+e_i)
+ \int_{B_R \setminus \mathcal I_R} |\nabla \varphi_{R,i}+e_i|^2
\\
&\qquad\qquad&-
2\int_{B_R \setminus \mathcal I_R} (\nabla \varphi_{R,i}+e_i) \cdot (\nabla \varphi_{R,K,i}+e_i).
\\
&\qquad\qquad=&E_{R,K,i}+E_{R,i} - 2\int_{B_R \setminus \mathcal I_R} (\nabla \varphi_{R,i}+e_i) \cdot (\nabla \varphi_{R,K,i}+e_i).
\end{eqnarray*}
Since $\nabla \varphi_{KRi} \cvf \nabla \varphi_{Ri}$ in $\Ld^2(B_R)$ and since minimal energies converge, this entails 
\[\limsup_{K \uparrow\infty} \int_{B_R} |\nabla (\varphi_{R,K,i}-\varphi_{R,i})|^2
\,\le\, 2E_{R,i}-2 \int_{B_R \setminus \mathcal I_R} |\nabla \varphi_{R,i}+e_i|^2\,=\,0,\]
which proves the strong convergence $\varphi_{R,K,i}\to \varphi_{R,i}$ in $H^1_0(B_R)$.
In addition, by elliptic regularity on $B_R \setminus\mathcal I_R$ up to the boundary, since $\varphi_{R,K,i}-\varphi_{R,i}$ is harmonic in $B_R\setminus\Ic_R$, the strong convergence in $H^1_0(B_R)$ actually implies convergence in $C^1(\overline{B_R}\setminus\mathcal I_R)$.

\medskip
\substep{1.2} Proof that $\nabla\varphi_{R,i}\to\nabla\varphi_i$ in $\Ld^2_\loc(\R^2)$ and $\varphi_{R,i}-\varphi_{R,i}(0)\to\varphi_i$ in $C^1_\loc(\R^2\setminus\Ic)$ as $R\uparrow\infty$ almost surely.\\
We proceed as in the proof of Theorem~\ref{th:diffeo}.
For $R,K\gg1$, let $\eta_{R,K}$ be a cut-off function for $B_{R(1-1/K)}$ in $B_R$, that is, such that
\begin{gather*}
\eta_{R,K}=1~~\text{on $B_{R(1-1/K)}$},\qquad \eta_{R,K}=0~~\text{on $\partial B_R$},\qquad |\nabla \eta_{R,K}|\lesssim KR^{-1},\\
\text{and $\eta_{R,K}$ is constant in $I_n^+:=I_n+\tfrac14\rho B$, for all $n$,}
\end{gather*}
where we recall that $\rho$ stands for the constant in the hardcore assumption, $\dist(I_n,I_m)\ge\rho$ almost surely for all $n\ne m$.
As~$\varphi_{R,i}$ satisfies equation~\eqref{eq:varphi-imperm-re-R}, we note that it satisfies the following relation in the weak sense on~$B_R$,
\[-\triangle \varphi_{R,i} = -\sum_{n \in N_R} \delta_{\partial I_n} (\nabla \varphi_{R,i}+e)\cdot \nu.\]
Noting that a similar relation holds for $\varphi_i$, a direct calculation then shows that the difference $\xi_{R,K,i}:=\eta_{R,K} \varphi_i-\varphi_{R,i}$ satisfies in the weak sense on $B_R$,
\begin{multline*}
-\triangle \xi_{R,K,i}\,=\,-\nabla \cdot (\varphi_i \nabla \eta_{R,K})+\nabla \cdot ((1-\eta_{R,K}) \nabla \varphi_i)\\
-\sum_{n \in N_R} \delta_{\partial I_n} \partial_\nu (\varphi_i-\varphi_{R,i})- \sum_{n \notin N_R}\delta_{\partial I_n \cap B_R} (\nabla \varphi_i+e_i)\cdot \nu.
\end{multline*}
Testing this equation with $\xi_{R,K,i} \in H^1_0(B_R)$, we get
\begin{multline}\label{eq:estim-psi-RKi}
\int_{B_R} |\nabla \xi_{R,K,i}|^2\,\lesssim\, K^2 \int_{B_R} R^{-2}|\varphi_i|^2  + \int_{B_R \setminus B_{R(1-1/K)}}|\nabla \varphi_i|^2\\
+\sum_{n\in N_R}\Big|\int_{\partial I_n}\xi_{R,K,i}\,\partial_\nu(\varphi_i-\varphi_{R,i})\Big|
+\sum_{n \notin N_R} \Big| \int_{\partial I_n \cap B_R} \xi_{R,K,i} (\nabla \varphi_i+e_i)\cdot \nu\Big|.
\end{multline}
The last two right-hand side terms require further investigation and we start with the first one.
First recall that the hardcore assumption ensures that the enlarged inclusions $I_n^+=I_n+\frac14\rho B$ are pairwise disjoint, which entails in particular that $\varphi_i-\varphi_{R,i}$ is harmonic in $I_n^+\setminus I_n$. Also recall that by definition $\xi_{R,K,i}-(1-\eta_{R,K})(x_i-\fint_{I_n^+}x_i)$ is constant on $I_n$ and that we have
$\int_{\partial I_n} \partial_\nu  (\varphi_i-\varphi_{R,i})=0$ for all $n\in N_R$. 
Integrating by parts and recalling that $\eta_{R,K}$ is chosen to be constant in $I_n^+$, we then find for $n\in N_R$,
\begin{eqnarray*}
\int_{\partial I_n}\xi_{R,K,i}\,\partial_\nu(\varphi_i-\varphi_{R,i})
&=&\int_{\partial I_n}(1-\eta_{R,K})\Big(x_i-\fint_{I_n^+}x_i\Big)\,\partial_\nu(\varphi_i-\varphi_{R,i})\\
&=&-\int_{I_n^+\setminus I_n}(1-\eta_{R,K})\nabla\Big(\Big(x_i-\fint_{I_n^+}x_i\Big)\chi_n\Big)\cdot\nabla(\varphi_i-\varphi_{R,i}),
\end{eqnarray*}
where $\chi_n$ is any smooth cut-off function for $I_n$ in $I_n^+$, that is, such that
\[\chi_n=1~~\text{in $I_n$},\qquad\chi_n=0~~\text{outside $I_n^+$},\qquad \text{and}\qquad|\nabla\chi_n|\lesssim1.\]
Hence, for $n\in N_R$,
\[\Big|\int_{\partial I_n}\xi_{R,K,i}\,\partial_\nu(\varphi_i-\varphi_{R,i})\Big|
\,\lesssim\,\int_{I_n^+\cap(B_R\setminus B_{R(1-1/K)})}|\nabla(\varphi_i-\varphi_{R,i})|,\]
and thus, decomposing $\varphi_i-\varphi_{R,i}=(1-\eta_{R,K})\varphi_i+\xi_{R,K,i}$,
\begin{multline}\label{eq:estim-psi-RKi-rhs1}
\Big|\int_{\partial I_n}\xi_{R,K,i}\,\partial_\nu(\varphi_i-\varphi_{R,i})\Big|
\,\lesssim\,
|I_n^+\cap(B_R\setminus B_{R(1-1/K)})|^\frac12\\
\times\bigg(\int_{I_n^+\cap(B_R\setminus B_{R(1-1/K)})}|\nabla\varphi_i|^2+\int_{I_n^+\cap B_R}|\nabla\xi_{R,K,i}|^2+K^2R^{-2}|\varphi_i|^2\bigg)^\frac12.
\end{multline}
We turn the last right-hand side term in~\eqref{eq:estim-psi-RKi}.
Since $\xi_{R,K,i}=0$ on $\partial B_R$, we may implicitly extend it by $0$ outside $B_R$.
As $\int_{\partial I_n}(\nabla\varphi_i+e_i)\cdot\nu=0$ and as $\varphi_i$ is harmonic in $I_n^+\setminus I_n$, we may then compute
\begin{eqnarray*}
\int_{\partial I_n \cap B_R} \xi_{R,K,i} (\nabla \varphi_i+e_i)\cdot \nu
&=&\int_{\partial I_n} \Big(\xi_{R,K,i}-\fint_{I_n^+}\xi_{R,K,i}\Big) (\nabla \varphi_i+e_i)\cdot \nu\\
&=&-\int_{I_n^+\setminus I_n} \nabla\Big(\Big(\xi_{R,K,i}-\fint_{I_n^+}\xi_{R,K,i}\Big)\chi_n\Big) \cdot(\nabla \varphi_i+e_i),
\end{eqnarray*}
where $\chi_n$ is as above.
Hence, by the Cauchy--Schwarz inequality and by Poincar\'e's inequality in $I_n^+$, we deduce
\begin{equation*}
\Big|\int_{\partial I_n \cap B_R} \xi_{R,K,i} (\nabla \varphi_i+e_i)\cdot \nu\Big|
\,\lesssim\,\Big(\int_{I_n^+\cap B_R} |\nabla\xi_{R,K,i}|^2\Big)^\frac12\Big(\int_{I_n^+}|\nabla \varphi_i|^2+1\Big)^\frac12.
\end{equation*}
Inserting this together with~\eqref{eq:estim-psi-RKi-rhs1} into~\eqref{eq:estim-psi-RKi}, recalling that the enlarged inclusions $\{I_n^+\}_n$ are disjoint, and using Young's inequality to absorb part of the right-hand side of \eqref{eq:estim-psi-RKi-rhs1} into the left-hand side of \eqref{eq:estim-psi-RKi}, we get
\begin{multline}
\int_{B_R} |\nabla \xi_{R,K,i}|^2\,\lesssim\, K^2 \int_{B_R} R^{-2}|\varphi_i|^2  + \int_{B_R \setminus B_{R(1-1/K)}}(|\nabla \varphi_i|^2+1)\\
+\sum_{n \notin N_R} \int_{I_n^+}(|\nabla \varphi_i|^2+1).
\end{multline}
Now arguing similarly as in Step~1 of the proof of Theorem~\ref{th:diffeo}, this implies almost surely 
\[\limsup_{K\uparrow\infty}\limsup_{R\uparrow\infty}\fint_{B_R}|\nabla\xi_{R,K,i}|^2\,=\,0,\]
and thus, by definition of $\xi_{R,K,i}$,
\[\lim_{R\uparrow\infty}\fint_{B_{R/2}}  |\nabla (\varphi_i-\varphi_{R,i})|^2\,=\,0.\]
Next, we appeal to the qualitative large-scale Lipschitz regularity for rigid inclusions
of~\cite[Theorem~2]{MR4468195} (the latter is for the  Stokes equation with rigid inclusions, but the present scalar case with stiff inclusions is similar and simpler to prove). As above, this entails that there exists an almost surely finite random radius $r_*\ge1$ such that for all $\rho\ge r_*$,
\[\lim_{R\uparrow\infty}\int_{B_{\rho}}  |\nabla (\varphi_i-\varphi_{R,i})|^2\,=\,0.\]
This proves the strong convergence $\nabla\varphi_{R,i}\to\nabla\varphi_i$ in $\Ld^2_\loc(\R^2)$ almost surely.
By elliptic regularity on $\R^2\setminus \Ic$ up to the boundary, since $\varphi_{R,i}-\varphi_i$ is harmonic in $\R^2\setminus\Ic$, this actually implies convergence $\varphi_{R,i}-\varphi_{R,i}(0)\to\varphi_i$ in $C^1_\loc(\R^2\setminus\Ic)$.

\medskip
\step{2} Proof that $\Phi(x):=x+\varphi(x)$ is a local diffeomorphism on $\R^2\setminus\mathcal I$ almost surely.\\
We appeal to the approximation devised in Step~1: consider the approximate corrected coordinates
\[\Phi_{R}(x)=x+\varphi_{R}(x)-\varphi_{R}(0),\qquad\Phi_{R,K}(x)=x+\varphi_{R,K}(x)-\varphi_{R,K}(0),\qquad\text{on $B_R$,}\]
with $\varphi_{R}=(\varphi_{R,1},\varphi_{R,2})$ and $\varphi_{R,K}=(\varphi_{R,K,1},\varphi_{R,K,2})$.
By~\cite[Theorem~4]{Alessandrini-Nesi-01} and \cite[Theorem~1.1]{Alessandrini-Nesi-18}, since the coefficient field $\Aa_{R,K}$ is uniformly elliptic, bounded, and H\"older-continuous, the map $\Phi_{R,K}$ is known to be a diffeomorphism from $B_R$ to itself.
Since $\Phi_{R,K}$ and $\Phi_R$ are both harmonic on $B_R\setminus\mathcal I_R$, they are given by the real part of a holomorphic function on that domain. Since the number of critical points of a holomorphic function in a domain is given by a contour integral involving its gradient, it is continuous with respect to $C^1$ convergence. As $\Phi_{R,K}$ converges to $\Phi_R$ in $C^1_b(B_R \setminus\Ic_R)$, cf.~Step~1.1, the non-existence of critical points of  $\Phi_{R,K}$ in~$B_R$ entails the non-existence of critical points of $\Phi_R$ in $B_R \setminus\Ic_R$. As an harmonic function without critical points, $\Phi_R$ is a local diffeomorphism on $B_R \setminus\Ic_R$.
Similarly, using the convergence of $\Phi_R$ to $\Phi$ almost surely, cf.~Step~1.2, we can conclude that $\Phi$ is a local diffeomorphism almost surely on $\R^2 \setminus \Ic$.

\medskip
\step3 Almost sure injectivity of $\Phi$ on $\R^2\setminus\Ic$.\\
Since the map $\Phi_{R,K}$ is a diffeomorphism from $B_R$ to itself by Step~2, we have the following change-of-variable formula: for all measurable subsets $E\subset B_R$ with $\dist(E,\partial B_R)\ge2$ and all locally integrable functions~$f$,
\begin{equation*}
\int_{\Phi_{R,K}(E \setminus \Ic_R)} f(y)\,dy \,=\,\int_{E\setminus\Ic_R} f(\Phi_{R,K}(x))\, |\!\det \nabla \Phi_{R,K}(x)|\,dx.
\end{equation*}
Passing to the limit in this identity, using the convergence $\varphi_{R,K,i}\to \varphi_{R,i}$ in $C^1_b({B_R}\setminus\Ic_R)$ as $K\uparrow\infty$, cf.~Step~1.1, and the convergence $\varphi_{R,i}-\varphi_{R,i}(0)\to \varphi_{i}$ in $C^1_\loc(\R^2\setminus\Ic)$ as $R\uparrow\infty$ almost surely, cf.~Step~1.2, we can deduce almost surely, for all measurable subsets $E\subset B_R$ with $\dist(E,\partial B_R)\ge2$ and all locally integrable functions~$f$,
\begin{equation}\label{e.chg-var++}
\int_{\Phi(E\setminus \mathcal I)} f(y)\,dy \,=\,\int_{E\setminus\mathcal I} f(\Phi (x))\, |\!\det \nabla \Phi(x)|\,dx.
\end{equation}
This  implies the  injectivity of $\Phi$ on $\R^2 \setminus \mathcal I$.
Indeed, assume that there exist $x,y \in \R^2 \setminus \Ic$, $x\ne y$, such that $\Phi(x)=\Phi(y)=z$. Since by Step~2 the map $\Phi$ is a local diffeormorphism both at $x$ and at $y$, there exists a ball $B_\e(z)$ of radius $\e>0$ centered at $z$ such that every point of $B_\e(z)$ has at least two antecedents in $\R^2 \setminus \mathcal I$ via $\Phi$. The associated area formula, e.g.~\cite[Section~3.3.2, Theorem~1]{MR3409135}, would then contradict \eqref{e.chg-var++}. This concludes the proof that $\Phi$ is injective on $\R^2 \setminus \mathcal I$ almost surely.
\end{proof}

\section*{Acknowledgements}
The authors thank Matthieu Hillairet for preliminary discussions on the topic.
MD acknowledges financial support from the F.R.S.-FNRS, as well as from the European Union (ERC, PASTIS, Grant Agreement n$^\circ$101075879). AG acknowledges financial support from the European Research Council (ERC) under the European Union's Horizon 2020 research and innovation programme (Grant Agreement n$^\circ$864066). Views and opinions expressed are however those of the authors only and do not necessarily reflect those of the European Union or the European Research Council Executive Agency. Neither the European Union nor the granting authority can be held responsible for them.

\bibliographystyle{abbrv}
\bibliography{biblio}

\end{document}